\documentclass{article}
\usepackage[a4paper,margin=3.2cm]{geometry}
\usepackage[T1]{fontenc}
\usepackage[applemac]{inputenc}

\usepackage{amsfonts,amsthm,amssymb,amsmath, mathdots, bbm}
\usepackage{amssymb}
\usepackage[pdftex]{color,graphicx}

\usepackage[extra]{tipa}

\numberwithin{equation}{section}

\title{Uniform regular weighted graphs with large degree: Wigner's law, asymptotic freeness and graphons limit}
\author{Camille Male, Sandrine P\'ech\'e}
\date{}

\newtheorem{Th}{Theorem}[section]

\newtheorem{Def}[Th]{Definition}
\newtheorem{Prop}[Th]{Proposition}

\newtheorem{Lem}[Th]{Lemma}

\newtheorem{Ass}{Assumption}

\newtheorem{Rem}{Remark}

\renewcommand\leq\leqslant

\renewcommand\geq\geqslant

\def\Tr{\mathrm{Tr}}

\def\esp{\mathbb E}

\def\etc{,\ldots ,}

\def\limN{\underset{N \rightarrow \infty}\longrightarrow}
\def\Nlim{\underset{N \rightarrow \infty}\lim}

\def\eps{\varepsilon}

\def\eq{\begin{eqnarray*}}
\def\qe{\end{eqnarray*}}
\def\eqa{\begin{eqnarray}}
\def\qea{\end{eqnarray}}

\def\mbf{\mathbf}
\def\mcal{\mathcal}
\def\mbb{\mathbb}
\def\trm{\textrm}
\def\mrm{\mathrm}

\newcommand{\one}{\mathbbm{1}}
\newcommand{\brem}{\begin{Rem}}
\newcommand{\erem}{\end{Rem}}

\begin{document}
\maketitle

\begin{center}
\begin{minipage}{14cm}
\begin{center}{\sc abstract:}\end{center}
{
For each $N\geq 1$, let $G_N$ be a simple random graph on the set of vertices $[N]=\{1,2, \ldots, N\}$, which is invariant by relabeling of the vertices. The asymptotic behavior as $N$ goes to infinity of correlation functions:
	$$ \mathfrak C_N(T)= \esp\bigg[ \prod_{(i,j) \in T} \Big( \one_{\big( \{i,j\} \in G_N \big)} - \mbb P(   \{i,j\} \in G_N) \Big)\bigg], \ T \subset [N]^2 \trm{ finite}$$
furnishes informations on the asymptotic spectral properties of the adjacency matrix $A_N$ of $G_N$. 
Denote by $d_N = N\times \mbb P(  \{i,j\} \in G_N) $ and assume $d_N, N-d_N\limN \infty$. If $\mathfrak C_N(T) =\big( \frac{d_N}N\big)^{|T|} \times O\big( d_N^{-\frac {|T|}2}\big)$ for any $T$, the standardized empirical eigenvalue distribution of $A_N$ converges in expectation to the semicircular law and the matrix satisfies asymptotic freeness properties in the sense of free probability theory. We provide such estimates for uniform $d_N$-regular graphs $G_{N,d_N}$, under the additional assumption that $|\frac N 2 - d_N- \eta \sqrt{d_N}| \limN \infty$ for some $\eta>0$. 
Our method applies also for simple graphs whose edges are labelled by i.i.d. random variables.
}

\end{minipage}
\end{center}

\section{Introduction and statement of results}
The article is concerned with  some statistical spectral properties of certain symmetric random matrices constructed from random graphs on $N$ vertices. We are interested in the global asymptotic behavior when $N$ is large of the (mean) empirical spectral distribution (e.s.d.) of such matrices $M_N$ as well as in questions coming from free probability. Implicitly, whenever we consider a random matrix $M_N$ or a random graph $G_N$, we actually mean a sequence $(M_N)_{N\geq 1}$ (resp. $(G_N)_{N\geq 1}$), where for each $N\geq 1$, $M_N$ is a random matrix of size $N \times N$ (resp. $G_N$ is a random graph) on the set of vertices $[N]:=\{1\etc N\}$.

We can now define the model under study.
Let $G_N$ be an undirected random graph on $[N]$. To each edge $\{i,j\} \in G_N$ is assigned the real weight $\xi_{i,j} = \xi_{j,i}$. We make the following hypotheses.
 \begin{Ass}\label{Ass:1}
\begin{enumerate}
	\item $G_N$ is simple, i.e. has no loops and each edge is of multiplicity one. 
	\item It is invariant in law by relabeling of its vertices, that is any weighted graph obtained by permuting the vertices of $G_N$ has the same distribution.
	\item the $\xi_{i,j}$, $1\leq i\leq j\leq N$, are independent identically distributed (i.i.d.) random variables, not identically zero and which admit moments of any order (i.e. $0<\esp\big[|\xi_{i,j}|^{2K}\big]<\infty$ for any $K\geq 1$). 
\end{enumerate}
\end{Ass}
Most of our efforts in this article are dedicated to the uniform regular graph, that is when the random graph is chosen uniformly among the set of graphs in $[N]$ with given degree $d_N$. We here assume that $d_N$ goes to infinity with $N$. More generally our study extends to any random weighted graphs $G_N$ satisfying the three above hypothesis.
\\The weighted adjacency matrix $A_N = (A_N{(i,j)})_{i,j=1\etc N}$ of $G_N$ is the symmetric matrix of size $N$ defined by: for any $i,j=1\etc N$
			$$A_N(i,j) = \left\{ \begin{array}{cc} 
						\xi_{i,j} & \textrm{if } \{i , j\} \textrm{ is an edge of } G_N, \\
						0 & \textrm{otherwise.}
					\end{array} \right.$$
Denote by $\lambda_1, \ldots, \lambda_N$ the eigenvalues of a symmetric matrix $A_N$. Then its spectral measure $\mu_{A_N}$ is the probability measure 
	$$\mu_{A_N} =   \frac 1 N \sum_{i=1}^N \delta_{\lambda_i}.$$ 					
Note that $A_N$ is invariant by conjugation by permutation matrices. The ''standardized`` version of the weighted adjacency matrix of $G_N$ is the random matrix
	\eqa\label{Eq:DefMN}
		M_N = \frac {  A_N - m_1  \frac{d_N}{N-1} J_N }{ \sqrt{d_N(m_2-m_1^2  \frac{d_N}{N-1})}},
	\qea
where $J_N$ is the matrix whose extra diagonal entries are ones and diagonal entries are zero, $m_1=\esp[\xi_{i,j}]$, $m_2= \esp[\xi_{i,j}^2]$, and $d_N$, the mean degree of the graph, is defined by
	\eqa\label{Eq:DefdN}
		d_N := \esp\Big[ \sum_{j\in[N]\setminus\{i\}} \one_{ \{i,j\} \trm{ is an edge of } G_N}\Big], \ \forall i \in [N].
	\qea
In the definition \eqref{Eq:DefdN}, there is no dependence on $i$ since the graph is invariant by relabeling of its vertices. 
By standardized, we mean that $M_N$ has centered entries and that for the polynomial $P(x)=x^2$, one has $\esp\mu_{M_N}(P) = \esp\big[ \frac 1 N \Tr M_N^2\big]=1$. 
\\
\par Even in the case where the entries of $M_N$ are independent, a variety of so-called ''sparse'' weighted random graphs (when $d_N$ is bounded) exhibit spectral properties which differ from classical models of random matrix theory.\\
The Erd\"os-R\'enyi (E.R.) random graph $G(N,\alpha_N)$ is the undirected weighted random graph for which any pair of vertices are connected with probability $\alpha_N$, independently of the other pairs of vertices. It fits the assumptions above with $d_N= \alpha_N \times (N-1)$.
When $\alpha_N=\frac d {N}$ with $d> 0$ fixed, the e.s.d. of the standardized adjacency matrix converges weakly and in moments to a symmetric measure. Since the measures $\mu_{M_N}$ are random , it should be mentioned that this convergence is almost sure, in probability and in moments. Apart formulas for its moments, few is known about the limiting e.s.d. It always has unbounded support \cite{Khor}. When the $\xi_{i,j}$ are constant equal to one, it is known that it has a dense set of atoms if $d\leq 1$ and it is conjectured that it has a density otherwise. The value of the mass at zero has been computed recently \cite{just}. \\
If one considers now an E.R. random graph with parameter $\alpha_N$ converging to some constant $\alpha \in ]0,1[$, then the standardized adjacency matrix is a special case of Wigner random matrices. Recall that a Wigner matrix is a random symmetric matrix $X_N=(X(i,j))_{i,j=1\etc N}$ with independent, centered entries, such that the diagonal (resp. extra-diagonal) entries of $\sqrt N X_N$ are distributed according to a measure $\nu_0$ (resp. $\nu$) that does not depend on $N$. Assume that $\int x^{2+\eps}\trm d\nu$ and $\int x^{2+\eps}\trm d\nu_0$ are finite. Then the e.s.d. of a Wigner matrix converges $^*$-weakly to the semicircular law $\mu^{(a)}_{\trm{S.C}}$  with radius $2a$, that is
	$$\mu^{(a)}_{\trm{S.C.}}(\trm d t) = \frac 1 {2\pi a^2} \sqrt{ 4a^2 - t^2} \one_{|t| \leq 2a}, \quad a^2=\int t^2 \nu(dt).$$
The convergence for the standardized E.R. random weighted graph to the semicircular law actually holds as soon as $\alpha_N = \frac{d_N}{N}$ where both $d_N$ and $N-d_N$ tend to infinity with $N$. 
The limitation that $N-d_N$ goes to infinity is necessary. When $\alpha_N=\frac {N-1-d} {N-1}$, $d>0$ fixed, the e.s.d. of $G(N,\alpha_N)$ can be related to the sparse case where $\alpha_N=\frac{d}{N-1}$. To see that point, consider the complementary graph of $G(N,\frac {N-1-d} {N-1})$, for which two vertices are connected by an edge iff they are not connected in $G(N,\frac {N-1-d} {N-1})$. Then the complementary graph is distributed as $G(N,\frac{d}{N-1})$. When the weights $\xi_{i,j}$ are constant, the standardized adjacency matrix of $G(N,\frac {N-1-d}{N-1})$ is distributed as the opposite of the standardized adjacency matrix of $G(N, \frac{d}{N-1})$.
\\
\par The random regular graphs $G_{N,d_N}$ have also been extensively studied in the literature (\cite{McKay}, \cite{DumiPal},\cite{Dumitriu2} \cite{BAD}). Let $N\geq 2$ and $1\leq d_N\leq N-1$ be given integers such that the product $d_N \times N$ is even. Recall that $G_{N,d_N}$ is a random graph chosen uniformly from the set of $d_N$-regular graphs. Let $A_{N}$ be the adjacency matrix of the uniform random regular graph with degree $d_N$ (unweighted, i.e. $\xi_{i,j}=1$ for any $i,j$). When $d$ is fixed independently of $N$, the e.s.d. of $A_N$ converges weakly and in moments to the Mc Kay \cite{McKay} law, i.e. the distribution with density
		$$\mu_{\trm{M.K.}(d)} (\trm{d}x) = \frac{ d\sqrt{ 4(d-1) -x^2}}{2\pi (d^2-x^2)} \one_{|x|\leq 2 \sqrt{ d-1}}\trm d x.$$
If the graphs are weighted by random variables, it is still true that the e.s.d. converges but the distribution depends on the common distribution of the $\xi_{i,j}$.
\\ Tran, Vu and Wang \cite{TranVuWang} consider the case where both $d_N$ and $N-d_N$ go to infinity with $N$. They  show the weak convergence of the e.s.d. of the standardized matrix $M_N$ of the unweighted $d_N$ regular graph to the semicircular distribution $\mu_{SC}^{(1)} $. The convergence is almost sure and in probability. Moreover, a local version is proved: let $I\subset (-2,2)$ be an interval of length $d_N^{-1/10}\ln^{1/5} d_N$. Call $N_I$ the number of eigenvalues falling in $I$. Then $\mathbb{P} (|N_I-n \int_I \mu_{SC}^{(1)}(dx)|\geq \delta n \int_I \mu_{SC}^{(1)}(dx))$ decreases exponentially fast. The basic idea in \cite{TranVuWang} is that an E.R. graph with parameter $\alpha_N$ is "quite often" $d_N=(N-1)\times \alpha_N$-regular. More precisely the authors compare the probability that it is indeed $d_N$-regular with the speed of concentration of linear statistics for the E.R. graph. This allows them to extend the asymptotic behavior of linear statistics of eigenvalues valid for E.R. graphs (based on random matrix theory results) to random regular graphs.
Moreover Dumitriu and Pal \cite{DumiPal} prove, under the hypothesis that $d_N$ is smaller than any power of $N$ ($d_N=N^{o(1)}$), that the convergence in moments of $\mcal \mu_{M_N}$ holds true. The local law is also improved therein. Their method is based on a local approximation of the $d_N$ regular graph by the infinite $d_N$-regular tree.
\\
\par A first application of our main results brings a slight contribution for that picture with new methods.

\begin{Th}[Wigner's law]\label{Regesd} Let $M_{N,d_N}$ the matrix of a uniform $d_N$-regular weighted graph. Assume that $d_N$, $ N-d_N$ and $|\frac N 2-d_N -\eta\sqrt{d_N}|$ tend to infinity for some constant $\eta>0$. Then, the mean empirical spectral distribution of $M_N$ converges in moments to the semicircular distribution, namely
	\eq
		 \forall P \ \mathrm{polynomial}, \ \ \esp\left[ \frac 1 N \sum_{i=1}^N P(\lambda_i) \right] \limN \int_{-2}^2 P(t) \frac1 {2\pi} \sqrt{ 4-t^2} \mrm{d}t,
	\qe
where $\lambda_1 \etc \lambda_N$ denote the eigenvalues of $M_{N,d_N}$.
\end{Th}

\par We actually first obtain a criterion for the convergence to the semicircular law of a standardized matrix of a weighted graph $G_N$ (see next section). 
Applying this criterion for the regular uniform graph is a large part of the problem. We proceed by combinatorial manipulations using intensively the so-called \emph{switching-method} (see Section \ref{Sec: prooflemma}). The condition $|\frac N 2 -d_N- \eta\sqrt{d_N}|\limN \infty$ is a limitation of our approach. 
\\
\par Our second application is an extension of this convergence in the context of free probability theory. Thanks to results from \cite{MAL12}, \cite{MAL122}, a minor modification of the mentioned criterion implies a phenomenon called ''asymptotic freeness`` (Definition \ref{Def:Free} below). First, recall the notion of joint distribution of several random matrices $\mbf A_N=(A_N^{(j)})_{j\in J}$, defined when the entries of each matrix admit moments of any order. It is usually called the mean $^*$-distribution (where the symbol $^*$ refers to the conjugate transpose) and defined as the linear map 
	$$\Phi_{\mbf A_N}: P \mapsto \esp\Big[ \frac 1 N\Tr P(\mbf A_N) \Big],$$
 where the $P$ are $^*$-polynomials ($P(\mbf A_N)$ is a fixed finite complex linear combinations of words in the matrices $A_N^{(j)}$ and their transpose). Note that when the family consists only in one symmetric matrix, the $^*$-distribution amounts just to restrict the e.s.d. to polynomials.
 
\begin{Def}[Asymptotic freeness]\label{Def:Free} Families $\mbf A_N^{(1)} \etc \mbf A_N^{(n)}$ of random matrices are asymptotically free whenever,
\begin{enumerate}
	\item the joint family $\mbf A_N=(\mbf A_N^{(1)} \etc \mbf A_N^{(n)})$ has a limiting $^*$-distribution, that is $\Phi_{\mbf A_N}$ converges pointwise,
	\item  the limit of $\Phi_{\mbf A_N}$ is characterized in terms of the limiting $^*$-distributions of each $\mbf A_N^{(j)}$'s, thanks to the following formula. For $n\geq 1$, for any integers $i_1\neq i_2\neq i_3 \neq \dots \neq i_n$ and any $^*$-polynomials $P_1 \etc P_n$ such that $\Nlim \esp\big[ \frac 1 N \Tr P_{j}(\mbf A_N^{(i_j)}) \big] = 0$ for any $j=1\etc n$, one has 
	\eqa\label{Eq:Freeness}
		\Nlim \esp\Big[ \frac 1 N \Tr \big[ P_1(\mbf A_N^{(i_1)}) \dots P_n(\mbf A_N^{(i_n)}) \big] \Big]= 0.
	\qea
\end{enumerate}
\end{Def}

\par Freeness, which is defined as the relation \eqref{Eq:Freeness}, is considered as the analogue of classical independence in free probability. Voiculescu \cite{VOIC} and then Dikema \cite{DYK} proved that independent Wigner matrices $X_N^{(j)}, j\in J,$ are asymptotically free. Moreover, consider a family of deterministic matrices $\mbf Y_N$ converging in $^*$-distribution. Assume the matrices of $\mbf Y_N$ are uniformly bounded in operator norm. Then $X_N^{(j)}, j\in J,$ and $\mbf Y_N$ are asymptotically free (see \cite{AGZ}). The same holds for independent matrices $M_N^{(j)}, j\in J,$ of weighted E.R. graphs with $d^{(j)}_N$ and $N-d^{(j)}_N$ going to infinity for any $j\in J$ (by \cite[Corollary 3.9]{MAL12} with no modification of the proof).
\par Our main result is the following.

\begin{Th}[Asymptotic freeness] \label{RegAF} Let $M_N^{(j)}, j \in J$ be the adjacency matrices of independent  uniform $d_N^{(j)}$-regular weighted graphs ($d_N^{(j)}$ depends on $j\in J$). Assume that there exists $\eta >0$ such that for any $j\in J$, $d_N^{(j)},$ $ N-d_N^{(j)}$ and $|\frac N 2 - d_N^{(j)} -\eta \sqrt{d_N^{(j)}}|$ tend to infinity. Moreover, let $\mbf Y_N $ be a family of deterministic matrices, uniformly bounded in operator norm as $N$ goes to infinity and converging in $^*$-distribution. Then the matrices $M_N^{(j)}, j \in J$ and $\mbf Y_N$ are asymptotically free.
\end{Th}

To complete the comparison, let us mention that the sparse case does not behave exactly similarly to the E.R. case from the point of view of free probability. When $\alpha^{(j)}_N\sim\frac{d^{(j)}}{N}$ for $d^{(j)}>0$ fixed for any $j\in J$, E.R. matrices $M_N^{(j)}, j\in J$ are not asymptotically free \cite{MAL122}. A different notion of asymptotic freeness holds for $M_N^{(j)}, j\in J,$ and $\mbf Y_N$, called the asymptotic traffic-freeness \cite{MAL12}. We do not use so much this notion in this article. Just recall that the deterministic matrices $\mbf Y_N$ are required to satisfy stronger moment hypotheses. 
When $d^{(j)}_N\limN d^{(j)}>0$ for any $j\in J$ and the $\xi_{i,j}$ are constant equal to one, the adjacency matrices of uniform regular graphs $M_N^{(j)}, j\in J$ are still asymptotically free \cite[Section F]{MoWo}. Nevertheless, in general  $M_N^{(j)}, j\in J$ is not asymptotically free from deterministic matrices $\mbf Y_N$ converging in $^*$-distribution. Under the assumption that $\mbf Y_N$ converges in distribution of traffics, $M_N^{(j)}, j\in J$ and $\mbf Y_N$ are weakly asymptotically free in the sense of \cite{MAL12}.
\\

As a  byproduct of our method, which does not concern the spectral properties of the adjacency matrices $M_N$, we obtain an estimate related to the geometry of uniform regular graphs with large degree. We follow the terminology introduced by Lov\'asz \cite{Lovasz}. A random (unweighted) graph $G_N$ is said to converge in distribution of ''graphons`` if and only if for any fixed finite graph $T$ 
		$$\Nlim P(T \subset G_N) \trm{ exists.}$$
Note that for an E.R. graph with parameter $\alpha_N=\frac{d_N}N$, one has $P\big(T \subset G(N,\alpha_N)\big)=\big( \frac{d_N}N\big)^{n} $. This implies the convergence in distribution of ''graphons'' for such random graphs.

Consider a sequence $d_N$ for which there exists a constant $\eta>0$ such that $d_N, N-d_N$ and $|\frac N 2 -d_N- \eta\sqrt{d_N}|$ going to infinity as $N\to \infty$. 
\begin{Prop}[Graphons limit estimate]\label{Cor:Graphons} Let $G_{N,d_N}$ be a uniform regular graph where $d_N$ is a sequence as above. Then one has that 
	 \eqa\label{Eq:ShapeProfil}
		\mbb P(T \subset G_{N,d_N}) =\big( \frac{d_N}N\big)^{n} \Big(1+ o(d_N)^{-\frac12})\big).
	\qea
\end{Prop}
In particular, it implies the convergence in the sense of graphons of $G_{N,d_N}$ to the same limit as the one of the Erd\" os -R\'enyi random graph $G(N,\frac{d_N}N)$ when $\big( \frac{d_N}N\big)$ admits a limit in in $]0,1[\setminus \{\frac 1 2\}$.
\\
\par The roadmap of our proof is given in next section.

\section{Main arguments: the correlation functions}\label{Sec:Statement}
\par Hereafter, we consider a random matrix $M_N$ as in \eqref{Eq:DefMN}, constructed on a random graph $G_N$ satisfying Assumption \ref{Ass:1}: $G_N$ is a simple random graph invariant by relabeling of its vertices, weighted by i.i.d. random variables whose common distribution does not depend on $N$. The sequence $d_N$ is the mean degree of the graph defined in \eqref{Eq:DefdN}.
%
%
%
%
%
%
%
\par Let $T$ be a given simple graph with $n$ edges $e_1 \etc e_n$. Use integers to name the vertices of $T$ and let $N$ be larger than the maximum of these integers. Denote by $G_N(e_j)$ the indicator function of the event ''$e_j$ belong to $G_N$``. Note that $ \esp\big[ G_N(e_j) \big]=\frac{d_N}N=:\alpha_N$ for any $j$ and $ \esp\big[ \prod_{j=1}^n G_N(e_j) \big] =\mbb P(T\subset G_N)$. Define the function
		\eqa\label{Hyp2}
			\mathfrak C_N (T):=\esp\Big[ \prod_{i=1}^n \big( G_N(e_i) - \alpha_N \big) \Big].
		\qea
Heuristically, $\mathfrak C_N (T)$ measure some correlations between the edges of $G_N$.

\begin{Prop}\label{theo: esd}Assume that $d_N$ tends to infinity, and that either $ N-d_N$ tends also to infinity or that the random variables $\xi_{i,j}$ are not constant. Moreover, assume that for any simple graph $T$ with $n$ edges, one has $\mathfrak C_N (T) = \left( \frac{d_N}{N} \right)^n   \times  \eps_N(T)$ where 
	$$\eps_N(T) \to 0 \text{ when }n=2 \trm{ and is } o \big(\, d_N^{1-\frac n 2} \, \big) \trm{ for } n\geq 3.$$
 Then, the mean empirical spectral distribution of $M_N$ converges in moments to the semicircular distribution.
\end{Prop}

Proposition \ref{theo: esd} is proved in Section \ref{Sec:Method1}.

\begin{Prop} \label{theo: asyfree} Let $M_N^{(j)}, j\in J$ be independent matrices constructed as $M_N$ in Proposition \ref{theo: esd}. Assume again that $d_N$ tends to infinity, and that either $ N-d_N$ tends also to infinity or that the random variables $\xi_{i,j}$ are not constant. Moreover, assume that for any simple graph $T$ with $n$ edges, one has $\mathfrak C_N (T) = \left( \frac{d_N}{N} \right)^n   \times  \eps_N(T)$ where 
	$$\eps_N(T) \to 0 \text{  if }n=2 \trm{ and is } O \big(\, d_N^{\frac 12-\frac n 2} \, \big) \trm{ for } n\geq 3.$$
 Let $\mbf Y_N$ be a family of deterministic matrices, uniformly bounded in operator norm as $N$ goes to infinity and converging in $^*$-distribution. Then the $M_N^{(j)}, j\in J$ and $\mbf Y_N$ are asymptotically free (Definition \ref{Def:Free}).
%
%
%
%
\end{Prop}

The proof is given in Section \ref{Sec:Method11}. Propositions \ref{theo: esd} and \ref{theo: asyfree} obviously apply to the Erd\"os-R\'enyi graph for which $\mathfrak C_N(T)=0$ for any $T$. We also show these two theorems hold true for the uniform $d_N$-regular graph under slight assumptions below.
\\

Let $N\geq 2$ and $1\leq d_N\leq N-1$ be given integers such that the product $d_N \times N$ is even.
\begin{Th}\label{Prop:eps_N} Assume that $d_N, N-d_N$ and $|\frac N 2 -d_N- \eta\sqrt{d_N}|$ go to infinity as $N\to \infty$ for some constant $\eta$. Then, for the uniform $d_N$-regular graph $G_{N,d_N}$ and for any finite graph $T$ with $n$ edges,
	$$\mathfrak C_N(T) = \left( \frac{d_N}{N} \right)^n\times  O\big(d_N^{-\frac n 2}).$$
 In particular, Theorem \ref{Regesd} and Theorem \ref{RegAF} hold true.
\end{Th}
Theorem \ref{Prop:eps_N} is proved in Section \ref{Sec: prooflemma}. As Theorem \ref{Prop:eps_N} readily implies Proposition \ref{Cor:Graphons}, we here expose the proof. 
\begin{proof}[Proof of Proposition \ref{Cor:Graphons}] We recall that we assume Theorem \ref{Prop:eps_N} holds true.

Expending the product in the definition of $\mathfrak C_N(T)$ over the subgraphs $\tilde T$ of $T$, we get
	$$\mathfrak C_N(T)=\sum_{\tilde T \subset T} \big( -\frac{d_N}N\big)^{n-|\tilde T|}\mbb P(\tilde T \subset G_N),$$
where  $|\tilde T|$ denotes the number of edges of $\tilde T$. Note that Theorem \ref{Prop:eps_N} states that
	$$\mathfrak C_N(T) = \big( \frac{d_N}N\big)^{n}\times  O(d_N^{-\frac n 2}) .$$
 We reason by induction on the number of edges of $T$. If $|T|=1$, the statement is obvious. Assume now that $\mbb P(\tilde T \subset G_N)= \big( \frac{d_N}N\big)^{|\tilde T|} \times \big(1 + \tilde \eps_N(\tilde T) \big)$ where $\tilde \eps_N(\tilde T) = O(d_N^{-1 })$ for any $\tilde T$ with at most $n-1$ edges. With this notation, we have
	$$\mbb P(T \subset G_N) = \big( \frac{d_N}N\big)^{n}\times \Big(1- \sum_{\substack{\tilde T \subset T \\ \tilde T \neq T}}(-1)^{n-|\tilde T|} \tilde \eps_N(\tilde T)  +   O(d_N^{-\frac n 2}) \Big).$$
	This gives the desired estimate.
\end{proof}

\section{Proof of Proposition 2.1} \label{Sec:Method1}
Our method uses the so-called injective trace formalism and extended moment method developed in \cite{MAL12}. To explain this method, we need to define more graph theory notations. 
Hereafter, we consider finite non oriented graphs $T=(V, E)$, independent of $N$. Such graphs $T$ are to be seen as test functions that are evaluated at $G_N$ or at the associated matrix $M_N$. Loops and multiple edges are allowed for $T$. We then denote by $\overline{E}$ the collection of pairwise distinct edges of $T$. Then $|\overline{E}|$ counts the number of edges without multiplicity. 
A tree is a graph which is connected and has no cycle nor multiple edge. A fat tree is a graph $T$ such that $\overline{T}:=(V, \overline{E})$ is a tree. A double tree is a fat tree in which the multiplicity of each edge is $2$. We call a simple edge of a graph an edge of multiplicity one which is not a loop. Recall the notation $[N]=\{1\etc N\}$.

\paragraph{}
We here make the entries of the matrices $M_N$ explicit by writing $M_N = \big(M_N(i,j)\big)_{i,j=1\etc N}$. For any finite graph $T=(V,E)$ we set the function $\Tr^0$ (called the injective trace)
\begin{equation}\label{defTauN}
	\Tr_N^0\big[ T(M_N) \big] =  \sum_{ \substack{ \phi: V \to [N] \\ \trm{injective}}} \prod_{\{v,w\} \in E} M_N\big( \phi(v) , \phi(w) \big) .
 \end{equation}

Let $T_k$ be the graph consisting in a simple cycle with $k$ vertices $\{1\etc k\}$, with edges $\{1,2\} \etc \{k-1,k\}, \{k,1\}$. Given a partition $\pi$ of $\{1\etc k\}$, let $T_k^\pi $ be the graph (with possibly multiple edges and loops) obtained by identifying vertices in a same block. Formally, its set of vertices is $\pi$ and its multi-set of edges $E$ are given as follows: there is one edge between two blocks $V_i$ and $V_j$ of $\pi$  for each $n$ in $\{1\etc k\}$ such that $n\in V_i$ and $n+1 \in V_j$ (with notation modulo $k$). Then for any finite connected graph , for any $k\geq 1$ we have  
	\eqa\label{eq:TraceInjTrace}
		 \Tr \, M_N^k  = \sum_{\pi\in \mcal P(k)} \Tr^0\big[T_k^\pi(M_N)\big],
	\qea
where $\mcal P(k)$ is the set of partitions of $[k]$. The formula is obtained by classifying in the sum $\Tr M_N^k = \sum_{i_1\etc i_l} M_N(i_1,i_2) \dots M_N(i_k,i_1)$ which indices are equal or not, see \cite[Section 3.1]{MAL12}.
\par We say that a graph of the form $T=T_k^\pi$ for certains $k$ and $\pi$ is cyclic. A cyclic graph is a finite, connected graph such that there exists a cycle visiting each edge once, taking account into the multiplicity of the edges. By \eqref{eq:TraceInjTrace}, the convergence of the expectation of $ \frac 1N \Tr^0$ for any cyclic graph implies the convergence in moments of the e.s.d. of $M_N$. Our proof of Wigner's law (Proposition \ref{theo: esd}) is based on the following lemma.

\begin{Lem}\label{Lem:esd} Denote $\tau_N^0 = \esp \big[ \frac 1N \Tr^0 \big]$. For any cyclic graph $T$
	\eqa\label{Eq:CVtraffic}
		\tau_N^0\big[T(M_N)\big] \limN 
							\left\{\begin{array}{ccc}
						 1 & \trm{if } T \trm{ is a double tree}\\
						0 & \trm{otherwise}.
							\end{array} \right.
	\qea
\end{Lem}

\begin{proof}[Proof of Proposition \ref{theo: esd}] We here assume that Lemma \ref{Lem:esd} holds true. One then gets from \eqref{eq:TraceInjTrace} and \eqref{Eq:CVtraffic} that 
$$\esp  \Tr \, M_N^k  = \sum_{\pi\in \mcal P(k): T_k^\pi \text{ is a double tree}}1.$$
We may recall that the $k$-th moment of the semicircular law, given by the $k$-th Catalan number, is the number of oriented rooted trees with $\frac k 2$ edges (see \cite{GUI}). As all oriented rooted trees can be obtained from a partition $\pi$ in a unique way, this finishes the proof.\end{proof}


We split the proof of Lemma \ref{Lem:esd} into the cases where $T$ is a double tree or not, and state below Lemmas \ref{Lem:esd1} and \ref{Lem:esd2} accordingly. Let $T=(V,E)$ be a cyclic graph. Denote its edges by $e_1=(v_1,w_1) \etc e_n=(v_n,w_n)$ (with possible repetition due to multiplicity) and set
	\eq
		\delta^0_N\big[ T(M_N) \big] = \esp \Big[ \prod_{i=1}^n M_N\big( \phi(v_i), \phi(w_i) \big)\Big],
	\qe
where $\phi$ is any injection $V\to [N]$. Since the matrix is invariant by conjugation by permutation matrices, this quantity does not depend on $\phi$. 
 Then, by the definition of $\Tr^0$, namely Formula \eqref{defTauN}, we have 
\eq
	\tau_N^0\big[ T(M_N) \big] =  \frac 1 N \frac{ N!}{(N-|V|)!}\delta_N^0\big[ T(M_N) \big].
\qe
Denote $\alpha_N=\big(\frac{d_N}N\big)$ and recall that $  M_N = \frac {  A_N - m_1 \alpha_N J_N }{ \sqrt{ d_N(m_2-m_1^2  \alpha_N)}}$ with the notations in \eqref{Eq:DefMN}. Then, using the multi-linearity of $\delta_N^0$ (with respect to the edges of the graph), for any finite graph $T$ we get
	\eqa\label{Eq:EquationTr0}
		\tau_N^0\big[T(M_N)\big] &=& N^{|V|-1}{d_N}^{-\frac{ |E|}2} \times  \delta_N^0\bigg[T\Big(  \frac{ A_N - m_1 \alpha_N J_N}{\sqrt{m_2-m_1^2 \alpha_N}}    \Big)\bigg]\times\big(1+o(N^{-1})\big).
	\qea
For a given graph $T$, we denote by $s$ 	the number of simple edges of $T$.
We can then write that 
	\eqa\label{Eq:K_N}
		N^{|V|-1}{d_N}^{-\frac{ |E|}2} = \underbrace{N^{|V|-1 - |\bar E|} {d_N}^{|\bar E |-\frac{ |E|}2-\frac s2}}_{K_N} \times   d_N^{\frac s2} {\alpha_N}^{-|\bar E|}.
	\qea
We now appeal to the following classical lemma of graph theory (see \cite[Lemma 1.1]{GUI}).

\begin{Lem}\label{Lem:VEC} For a finite graph with $\mcal V$ vertices, $\mcal E$ edges and $\mcal C$ connected components, then $\mcal E + \mcal C - \mcal V$
is the number of cycles of the graph. This is the maximal number of edges we can suppress without disconnecting a component of the graph. In particular, 
	$$\mcal V-\mcal E - \mcal C  \leq 0$$
with equality if and only if the graph is a forest, i.e. its components are trees.
\end{Lem}
We first apply Lemma \ref{Lem:VEC} to the graph  $\bar T$ obtained by forgetting the multiplicity of the edges of $T$ so that $$\mcal V = |V|, \mcal E= |\bar E| \text{ and }\mcal C=1.$$ 
Hence, $K_N=1$ if and only if $T$ is a fat tree whose edges are of multiplicity one or two. When $T$ is a double tree, $K_N=1$ and $s=0$. Hence to prove the convergence $\tau_N^0\big[T(M_N)\big] \limN 1$ for double trees $T$, it then suffices to prove the

\begin{Lem}\label{Lem:esd1} For any double tree $T$
	\eqa
		\label{Eq:Delta1}
		\delta_N^0\bigg[T\Big(  \frac{ A_N - m_1\alpha_N J_N}{\sqrt{m_2-m_1^2\alpha_N}}    \Big)\bigg] & = &   \left( \frac{d_N}N\right)^{|\bar E|} \times  \big( 1+o(1)\big).
	\qea
\end{Lem}
 \begin{proof}[Proof of Lemma \ref{Lem:esd1}] Let $T=(V,E)$ be a double tree. Denote $e_1 \etc e_n$ the edges of $T$ without multiplicity (so that each edge is repeated twice in $T$). We want to prove that
	$$ \alpha_N^{-|\bar E|} \times \delta_N^0\bigg[T\Big(  \frac{ A_N - m_1\alpha_N J_N}{\sqrt{m_2-m_1^2\alpha_N}}    \Big)\bigg] =\alpha_N^{-n} (m_2-m_1^2\alpha_N)^{-n}  \times \esp\bigg[ \prod_{i=1}^n \big(  A_N(e_i) - m_1 \alpha_N  \big)^2 \bigg]$$
	converges to one. By the proof of Proposition \ref{Cor:Graphons} one has $\esp\big[ \prod_{i=1}^n G_N(e_i) \big]= \mbb P(T\subset G_N) = \alpha_N^{n}\big(1 + o(1) \big) = \prod_{i=1}^n \esp\big[ G_N(e_i) \big]\big(1 + o(1) \big)$ (we only use $\mathfrak C_N(T) = \alpha^n\big( 1+o(1)\big)$ for any $T$). Denote by $\esp\big[ \, \cdot \, \big | \, \xi \big]$ conditional expectation with respect to the algebra spanned by the i.i.d. random variables $\xi_{i,j}, 1\leq i< j\leq N$. Denote $\xi(e)=\xi_{i,j}$ for $e=\{i,j\}$. One has
	\eq
		   \lefteqn{\esp\bigg[ \prod_{i=1}^n \big(  A_N(e_i) - m_1 \alpha_N  \big)^2 \bigg] }\\
		& = &  \esp\bigg[   \esp\Big[  \prod_{i=1}^n \big( \xi_{e_i}G_N(e_i) - m_1 \alpha_N  \big)^2 \, \Big| \, \xi \Big] \ \bigg]\\
		    & = & \esp\bigg[  \prod_{i=1}^n  \esp\Big[ \xi_{e_i}^2G_N(e_i) -2m_1 \alpha_N \xi_{e_i}G_N(e_i)+ m_1^2 \alpha_N^2   \, \Big| \, \xi \Big]\times   \big(1+o(1)\big) \ \bigg]\\
		    & = &  \prod_{i=1}^n  \left( m_2 \alpha_N -m_1^2 \alpha_N^2 \right)\times  \big(1+o(1)\big) \\
		   &  = & \alpha_N^{n} (m_2-m_1^2\alpha_N)^{n}  \times \big(1+o(1)\big)
	\qe
as desired.
\end{proof}

Let us now consider a cyclic graph $T$ which is not a double tree. Since $T$ is cyclic, it is not possible that $T$ has an edge of order one or three and is simultaneously a fat tree. So either $T$ has an edge of multiplicity at least  four, or $\bar T$ admits a cycle. Hence for a cyclic graph which is not a double tree, in all cases one has that $$K_N = O( d_N^{-1})$$ (since $d_N\leq N-1$). By Formula \eqref{Eq:K_N} and Lemma \ref{Lem:VEC}, the fact that $\tau_N^0\big[T(M_N)\big] \limN 0$ when $T$ is not a double tree follows from the

\begin{Lem}\label{Lem:esd2} For any cyclic graph $T$ which is not a double tree, with $n$ edges and $s$ simple edges,
	\eqa
		\label{Eq:Delta2}
		 \delta_N^0\bigg[T\Big(  \frac{ A_N - m_1\alpha_N J_N}{\sqrt{m_2- m_1^2\alpha_N}}    \Big)\bigg]  & = &    \alpha_N^{|\bar E|} \times  o\big( d_N^{1-\frac s2}  \big).
	\qea
\end{Lem}

\begin{proof}[Proof of Lemma \ref{Lem:esd2}] By multi linearity, 
	$$\delta_N^0\bigg[T\Big(  \frac{ A_N - m_1\alpha_N J_N}{\sqrt{m_2- m_1^2\alpha_N}}    \Big)\bigg]  = (m_2-m_1^2\alpha_N)^{-\frac n2} \delta_N^0\big[T ( A_N - m_1\alpha_N J_N  )\big].$$ 
We claim that, without loss of generality, one can assume that $m_2-m_1^2\alpha_N$ is bounded away from zero. Indeed, remark first that by the Cauchy Schwarz inequality, one has $m_1^2\leq m_2$ with equality if and only if the random variables $\xi_{i,j}$ are constant. Hence, if the random variables are not constant, since $\alpha_N\leq 1,$ then $m_2-m_1^2\alpha_N$ is bounded away from zero. If the random variables are constant, one can assume this constant is one (this does not change the distribution of $M_N$). Moreover, recall we considered in the introduction the complementary graph (for which edges belong to the graph if and only if edges do not belong to the initial graph). For unweighted graphs, $-M_N$ is distributed as the matrix associated to the complementary graph and we can assume $d_N \leq \frac N 2$. Finally $m_2-m_1^2\alpha_N = 1-\alpha_N$ is always bounded away from zero.
\\
\par Hence it is sufficient to prove 
	\eqa\label{Eq:Delta2bis}
		\delta_N^0\big[T ( A_N - m_1\alpha_N J_N  )\big]   =    \alpha_N^{|\bar E|} \times  o \big( d_N^{1-\frac s2}  \big),
	\qea
where $s$ is the number of simple edges of $T$. Choose an enumeration of the edges of $T$ : $ E = \{e_1 \etc e_n\}$, with possible repetitions. Then one has that
\eqa\label{Eq:Lien}
	\delta_N^0\big[T ( A_N - m_1\alpha_N J_N  )\big]  = \esp \Big[ \prod_{i=1}^n \big(A_N(e_i) - m_1\alpha_N\big) \Big],
\qea
where $A_N(e_i)$ stands for the random variables $A_N(v_i,w_i)$ whenever $e_i=\{v_i,w_i\}$. 
\par If $T$ is a simple graph, then
	\eqa\label{Eq:RecallEps}
		\mathfrak C_N(T) := \esp \Big[ \prod_{i=1}^n \big(G_N(e_i) - \alpha_N\big) \Big] = \alpha_N^n \times \eps_N(T),
	\qea
where $G_N(e_i) = \one_{e_i \trm{ belongs to } G_N}$ and $\eps_N(T)=o(d_N^{1-\frac n2})$. Equation \eqref{Eq:RecallEps} gives the good estimates \eqref{Eq:Delta2bis} if $T$ is a simple graph. Indeed, in that case, since the $\xi_{i,j}$ are i.i.d., we obtain $\delta_N^0\big[T ( A_N - m_1\alpha_N J_N  )\big] = m_1 \times \mathfrak C_N(T) = o(d_N^{1-\frac n2})$. We then need to extend this fact for non simple graphs. We shall use the following lemma.

\begin{Lem}\label{Lem:EpskN} Let $T$ be a simple graph with edges $e_1 \etc e_n$. Then, for $k=0\etc n$, 
\eqa
	 \esp\Big[\prod_{i=1}^k G_N(e_i) \prod_{i=k+1}^{n} \big( G_N(e_i) - \alpha_N \big)  \Big]  = \alpha_N^n \big(\delta_{n,k} + \eps_N^{(k)}(T) \big),
\qea
where $\eps_N^{(k)}(T)$ tends to zero and $\eps_N^{(k)}(T)= o(d_N^{1-\frac {n-k}2 })$. Recall that the Kronecker symbol $\delta_{n,k}$ is one if $n=k$ and  zero otherwise. 
\end{Lem}

\begin{proof} Note that $\eps_N^{(0)}(T) = \eps_N(T)$, where $\eps_N(T)$ is defined in \eqref{Eq:RecallEps}, and $ \mbb P( T \subset G_N) = \alpha_N^n \times \big( 1 + \eps_N^{(n)} (T) \big)$. Writing $G_N(e_k) = \big(G_N(e_k)  - \alpha_N\big) + \alpha_N$ we obtain $\eps_N^{(k)}(T)  = \eps_N^{(k-1)}(T) + \eps_N^{(k-1)}(T\setminus{e_k}).$
%
%
By a simple recursion, $\eps_N^{(k)}(T)$ is a finite linear combination of $\eps_N(\tilde T)$ for subgraphs $\tilde T$ of $T$ with at least $n-k$ edges. Hence the result follows.
\end{proof}

We can now prove \eqref{Eq:Delta2bis}, using Formula \eqref{Eq:Lien} and Lemma \ref{Lem:EpskN}. Let $T$ be a simple graph with edges $e_1 \etc e_n$. Consider integers $p_j, j=1, 2, \ldots,n$ such that $p_j\geq2$ for $j=1\etc k$ and  $p_j=1$ for $j=k+1\etc n$. Let prove the estimate for the (non simple) graph obtained from $T$ by setting $p_j$ for the multiplicity of $e_j$. One has
\eqa
	\lefteqn{\esp\Big[ \prod_{i=1}^n \big(A_N(e_i) - m_1\alpha_N \big)^{p_j} \Big] }\\
	& = &  \sum_{ \substack{\ell_{1}  \in  [p_{1}]\\ \dots, \ell_k \in [p_k]}} \binom{p_1}{\ell_1}\dots \binom{p_k}{\ell_k} (-m_1\alpha_N)^{\ell_1 + \dots + \ell_k} \esp\Big[ \prod_{i=0}^{k} A_N(e_i)^{p_i-\ell_i} \times \prod_{i=k+1}^n \big(A_N(e_i) - m_1\alpha_N \big)\Big]\nonumber\\
	& = &  \sum_{ \substack{\ell_{1}  \in  [p_{1}]\\ \dots, \ell_k \in [p_k]}} \binom{p_1}{\ell_1}\dots \binom{p_k}{\ell_k} (-m_1\alpha_N)^{\ell_1 + \dots + \ell_k} \prod_{i=1}^k\esp[\xi]^{p_i-\ell_i}  m_1^{n-k}\nonumber\\
	& & \times  \esp\Big[ \prod_{i=0}^{k} G_N(e_i)^{p_i-\ell_i} \times \prod_{i=k+1}^n \big(G_N(e_i) - \alpha_N \big)\Big],\nonumber
	\qea
where $\xi$ is distributed as the $\xi_{i,j}$. Note that $G_N(e_i)^{m-\ell}=G_N(e_i)$ for any $\ell \neq m$. Let fix integers $\ell_j \in [p_j]$ and denote by $C$ is the number of $\ell_i$'s for which $p_i-\ell_i$ is nonzero. Then, by Lemma \ref{Lem:EpskN}, the associated expectation on the left hand side is of the form $\alpha_N^{C+(n-k)}\eps_N^{(C)}(\tilde T)$ for a graph $\tilde T$ with $C+(n-k)$ edges . For each $\ell_i=p_i$ we take an $\alpha_N$ from the $\ell_1 + \dots + \ell_k$, and, since $\alpha_N\leq 1$, we get 
\eq
	\esp\Big[ \prod_{i=1}^n \big(A_N(e_i) - m_1\alpha_N \big)^{p_j} \Big]= \alpha_N^n \times o( d_N^{1-\frac {n-k} 2}).
\qe
This gives the estimate \eqref{Eq:Delta2} and proves that $\tau_N^0\big[T(M_N)\big] \limN 0$ when $T$ is not a double tree.
\end{proof}

\paragraph{}
Hence we have proved Lemmas \ref{Lem:esd1} and \ref{Lem:esd2}, which gives Lemma \ref{Lem:esd} and finishes the proof of Proposition \ref{theo: esd}.

\section{Proof of Proposition 2.2}\label{Sec:Method11}

The proof of the Asymptotic freeness needs only slight modifications of the previous part, thanks to results from \cite{MAL12} and \cite{MAL122}. We use the notations of Lemma \ref{Lem:esd}. To establish Proposition \ref{theo: asyfree}, we use the following Lemma.

\begin{Lem}\label{Lem:DeduceAsyFree} Let $\mbf M_N =(M_N^{(j)})_{j=1\etc n}$ be a family of independent random matrices. Assume that each matrix is invariant in law by conjugation by permutation matrices, and that for any $j=1\etc n$
\begin{enumerate}
	\item for any cyclic graph $T$, $\tau_N^0\big[ T(M_N^{(j)})\big] \limN \one_{(T \trm{ is a double tree})} $,
	\item for any cyclic graphs $T_1 \etc T_k$, 
		$$\esp \Big[\prod_{i=1}^k \frac 1 N \Tr^0 \big[T_i(M_N^{(j)})\big]  \Big] - \prod_{i=1}^k \tau_N^0 \big[T_i(M_N^{(j)})\big] \limN 0.$$
	\item for any (possibly non-cyclic) graphs $T_1 \etc T_k$, 
		$$\esp \Big[\prod_{i=1}^k \frac 1 N \Tr^0 \big[T_i(M_N^{(j)})\big]  \Big] =O(1).$$
\end{enumerate}
Let $\mbf Y_N$ be a family of deterministic matrices converging in $^*$-distribution, whose matrices are uniformly bounded in operator norm. Then $M_N^{(1)} \etc M_N^{(n)}, \mbf Y_N$ are asymptotically free.
\end{Lem}

\begin{proof} The proof is identical to the proof of \cite[Corollary 3.9]{MAL122}. We recall briefly the idea. By \cite[Theorem 2.3]{MAL122}, the convergence of $\tau_N^0\big[T(M_N^{(j)})\big]$ for cyclic graphs $T$ and the two above assumptions implies that $M_N^{(1)} \etc M_N^{(n)}, \mbf Y_N$ satisfies a weak version of asymptotic freeness, under slighter stronger conditions on $\mbf Y_N$. In \cite{MAL12}, it is proved that when the limit of $\tau_N^0\big[T(M_N)\big]$ is the indicator of double trees as in \eqref{Eq:CVtraffic}, then this weak asymptotic freeness is actually the Voiculescu's asymptotic freeness. A tightness argument implies that $M_N^{(1)} \etc M_N^{(n)}, \mbf Y_N$ are asymptotically free in the sense of Voiculescu, with the above assumptions on $\mbf Y_N$. \end{proof}

We now show that the assumptions of Lemma \ref{Lem:DeduceAsyFree} are satisfied.
First Lemma \ref{Lem:esd} implies that the first assumption holds true. 
Let us now show how to modify the proof of the previous Section to prove the second assumption of Lemma \ref{Lem:DeduceAsyFree}. Denote the cyclic graphs by $T_i = (V_i,E_i)$ for any any $i=1\etc k$. Call $\mcal P(V_i)_{i=1\etc k}$ the set of partitions of $V_1 \sqcup \dots \sqcup V_k$ whose blocks contain at most one element of each $V_i$.
It is a matter of fact that
	\eqa\label{Eq:DiscoTrace}
		\prod_{i=1}^k \Tr^0\big[ T_i(M_N) \big] = \sum_{\sigma \in \mcal P(V_i)_{i=1\etc n}} \Tr^0 \big[ T_\sigma(M_N) \big],
	\qea
where $T_\sigma$ is the graph obtained from the disjoint union of the $T_i$'s by identifying vertices in a same block of $\sigma$. We get this formula by looking, for the injective maps in the injective traces, how their images intersect. Note that now $T_\sigma$ is possibly disconnected, but each of its connected component are cyclic.\\
Now, with no modification of the proof of \eqref{Eq:EquationTr0} in the previous part, we have: for any graph $T$ (which is considered as a $T_{\sigma}$) with $c$ connected components and $s$ simple edges,
	\eq
		\esp\Big[ \frac 1 {N^c} \Tr^0 \big[ T(M_N) \big] \Big] = \tilde K_N \times d_N^{\frac s 2} \alpha_N^{-|\bar E|} \delta_N^0\bigg[T\Big(  \frac{ A_N - m_1\alpha_N J_N}{\sqrt{1-\alpha_N}}    \Big)\bigg],
	\qe
where $\tilde K_N = N^{|V|-c-|\bar E|} d_N^{|\bar E|-\frac{|E|}2 - \frac s 2}$. By Lemma \ref{Lem:VEC}, the quantity $\tilde K_N$ is one if $T$ is a fat forest of double trees. Otherwise, since each component is cyclic, $\tilde K_N = O(d_N^{-1})$ with the same reasoning as in the previous part. The rest of the reasoning is as in the previous Section with no modification. Thus, in \eqref{Eq:DiscoTrace} the only term which is non negligible corresponds to  $\sigma$ being the trivial partition, and then
	$$\esp \big[\prod_{i=1}^k \frac 1 N \Tr^0 \big[T_i(M_N^{(j)})\big]  \big]\limN \one_{(\trm{all the } T_i\trm{'s are double trees})}.$$
Let us now indicate where we modify this proof in order to get the third assumption of Lemma \ref{Lem:DeduceAsyFree}. Remove in the above computations the fact that the $T_i$ are cyclic graphs. Hence the only detail that changes is now $\tilde K_N$ is $O(\sqrt{d_N}^{-1})$ if $T_{\sigma}$ is not a fat forest of double trees. This yields the result.
\section{Proof of Proposition 2.3} \label{Sec: prooflemma}

\par {\it Step 1: Preliminaries.} Let $G_{N}(=G_{N,d_N})$ be a uniform $d_N$-regular graphs on $[N]:=\{1\etc N\}$. Denote $\alpha_N=\frac{d_N}{N-1}$ and fix for the rest of the Section a simple graph $T=(V,E)$ with pairwise distinct edges $e_1, e_2, \etc e_n$. Let give an estimate for $\eps_N(T)$ defined below
	\eqa\label{Eq:Purpose1}
		\mathfrak C_N(T) \, \big( = \mathfrak C_{N,d_N}(T) \big)= \esp\Big[ \prod_{i=1}^n \big( G_N(e_i) - \alpha_N \big) \Big]=:\alpha_N^n \times \eps_N(T).
	\qea
Note that by considering the complementary graph $\mathfrak C_{N,d_N}(T) = (-1)^n\mathfrak C_{N,N-1-d_N}(T) $, so we can assume $\alpha_N\leq \frac 1 2$ without loss of generality.
\\
\par We use the combinatorial method, to expand the expectation with respect to the uniform measure on the set of all $d_N$-regular graphs. 
We give a first expression \eqref{Eq:Prelim} of the correlation function $\mathfrak C_N(T)$ and then explain the general idea of the proof. Denote by $\mcal G_N$ the set of all simple $d_N$-regular graphs on $[N]$.
We partition the graphs $G\in \mcal G_N$ by considering the set of edges of $T$ that belong to $G$. 
Given a subset $J\subset [n]:=\{1\etc n\}$, we denote by $\mcal G_N(J)$ the set of graphs $G \in \mcal G_N$ such that $e_j \subset G$ iff $j \in J$. 
For any $j=1\etc n$, if $e_j \subset G$ then $G(e_j) - \alpha_N = \frac{ N-1 - d_N}{N-1}$ and otherwise $G(e_j) - \alpha_N = \frac{ - d_N}{N-1}$. Hence we can write
	\eqa \label{Eq:CNT}
		\mathfrak C_N(T) =\sum_{J \subset [n]} (-1)^{n-|J|} \frac{ \sharp \mcal G_N(J) \times (N-1-d_N)^{|J|} d_N^{n - |J|}  }{\sharp \mcal G_N(J) \times (N-1)^n }.
	\qea

Remark that the quantity $ N-1 - d_N$ and $d_N$ respectively count the number of non-neighbors and neighbors $w'$ of a vertex $v$ in $G_N$. For each $j=1, \ldots,n$, disregarding coincidences, we write $$e_j=\{v_j,w_j\}$$ where $v_j,w_j$ are (fixed) vertices. Using this notation, we need to be able to distinguish the vertices $v_j$ and $w_j$, e.g. by orienting the edges $e_j$.
We now use the notations
	$$  v\sim_J w  \Leftrightarrow \big( v \sim w \trm{ if } j\in J \trm{ and } v \not\sim w \trm{ if } j\notin J \big),$$
	$$  v\not \sim_J w  \Leftrightarrow \big( v \not \sim w \trm{ if } j\in J \trm{ and } v \sim w \trm{ if } j\notin J \big).$$
We use the notation $\mcal G_N(J) \cup (w_j' \, | \,  w_j' \not\sim_J v_j )_{j\in [n]}$ for the set of graphs with labeled vertices $(G,w_j', j \in [n])$, where
\begin{enumerate}
	\item $G\in \mcal G_N(J)$, i.e. it is a $d_N$ regular graph and $v_j \sim_J w_j$ for any $j=1\etc n$ (the vertices $v_j$ and $w_j$ are fixed),
	\item $w_j' \in [N]$ is a vertex, which varies in the enumeration, such that $w_j' \not \sim_J v_j$.
\end{enumerate}
Starting with the graph $T$, we can write $\mathfrak C_N(T)$ as the signed sum 
	\eqa\label{Eq:Prelim}
		\mathfrak C_N(T) =\sum_{J \subset [n]} (-1)^{n-|J|} \frac{ \mcal G_N(J) \cup (w_j' \, | \,  w_j'\sim_J v_j)_{j\in [n]} }{\sharp \mcal  G_N   \times (N-1)^n   }.
	\qea
Indeed for any $j\in J$ (resp. $j \notin J$) we choose a vertex which is not (resp. is) a neighbor of $v_j.$	
The basic idea is as follows: fix some subset $J$. Consider now an index $j\in J$ so that $(v_j,w_j) \in G$ but $(v_j, w'_j)\notin G.$ Roughly speaking, switching these two edges and completing into a $d_N$-regular graph will result into 
almost the same number of graphs but with $J \to J\setminus\{j\}$. Therefore some cancellations will arise. However we need to be very precise if we want to stack up the different error terms for each $j\in J$. The plan is to pursue further this construction for each index $i\in J$ and build hexagons as in Figure \ref{Fig:5} below. There, $T$ is the square with four edges $e_1e_2e_3e_4$ and $J=\{1,2\}$. We anticipate in the figure below some notations defined later. We use a symmetry of the uniform regular graphs described here after to catch the simplifications in the sum \eqref{Eq:Prelim}.

  \begin{figure}[h!]
    \begin{center}\vspace*{-0.5cm}
     \includegraphics[width=110mm]{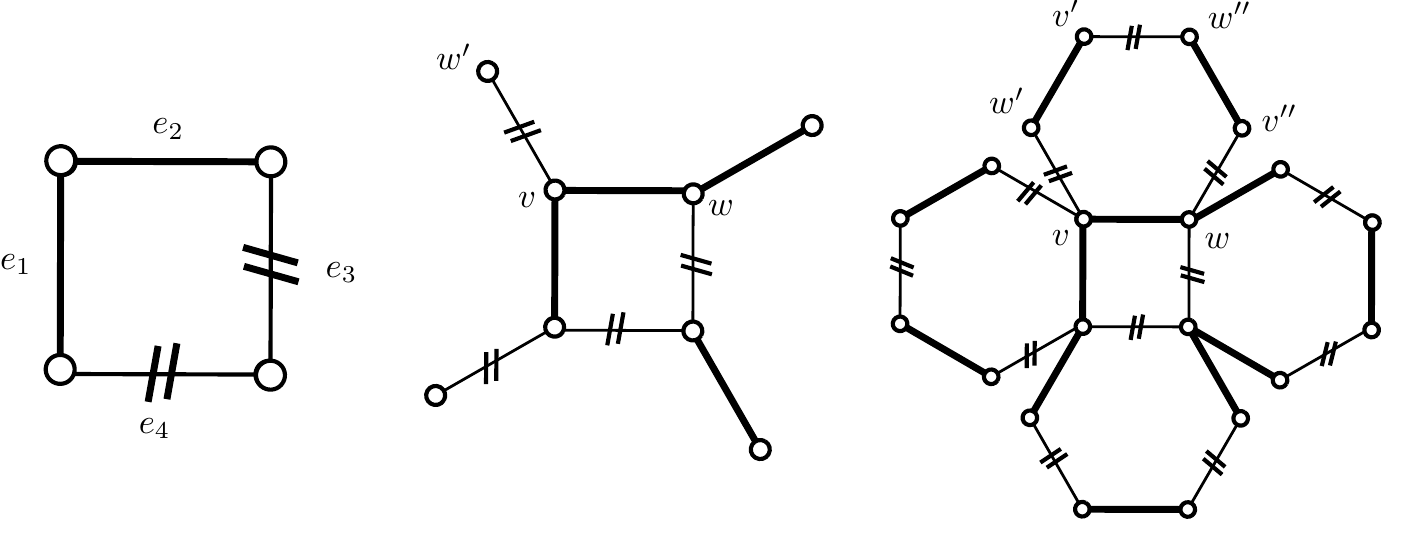}
    \end{center}
    \caption{Construction of the hexagons (the plan).}
    \label{Fig:5}
  \end{figure}

~
\\
\\ {\it The switching method:} Let $v^{(1)},w^{(1)} \etc v^{(K)},w^{(K)}$ be $2K$ distinct vertices. It is convenient to define the notation $w^{(K+1)}:=w^{(1)}$ and $v^{(K+1)}:=v^{(1)}.$ We also define the $2K$-gone $P$ obtained by drawing the edges 
$\{v^{(k)}, w^{(k)}\}$ and $\{w^{(k)}, v^{(k+1)}\}$, for all $k=1\etc K$.
Consider two disjoint sets of edges $\mcal E$ and $\mcal F$ that do not contains edges of the $2K$-gone $P$. Then, the number of $d_N$-regular graphs $G$ such that 
	\eqa\label{Switch1}
		v^{(k)}\sim w^{(k)}, \ w^{(k)} \not\sim v^{(k+1)} \ \forall k=1\etc K, \trm{ and } \ e\subset G \ \forall e \in \mcal E  \trm{ , } f\not \subset G \ \forall f \in \mcal F
	\qea
is equal to the number of the number of $d_N$-regular graphs $G$ such that 
	\eqa\label{Switch2}
		v^{(k)}\not\sim w^{(k)}, \ w^{(k)}  \sim v^{(k+1)} \ \forall k=1\etc K, \trm{ and } \ e\subset G \ \forall e \in \mcal E  \trm{ , } f\not \subset G \ \forall f \in \mcal F
	\qea
Indeed, denote by $\mcal G^{(1)}$ and $\mcal G^{(2)}$ the sets of regular graphs defined by \eqref{Switch1} and \eqref{Switch2} respectively. Consider a graph from $\mcal G^{(1)}$. By ''switching" its edges, we mean removing the edges $\{v^{(k)},w^{(k)}\}$'s and adding the $\{w^{(k)},v^{(k+1)}\}$'s. This switching maps bijectively the graph to an element from $\mcal G^{(2)}$. The important fact is that we do not modify the degree of any vertices when applying this switching.  Moreover, as one can easily verify, one can add additional assumptions on the sets $\mcal G^{(1)}$ and $\mcal G^{(2)}$ that are preserved by switching. We may use this property later.

  \begin{figure}[h!]
    \begin{center}
     \includegraphics[width=110mm]{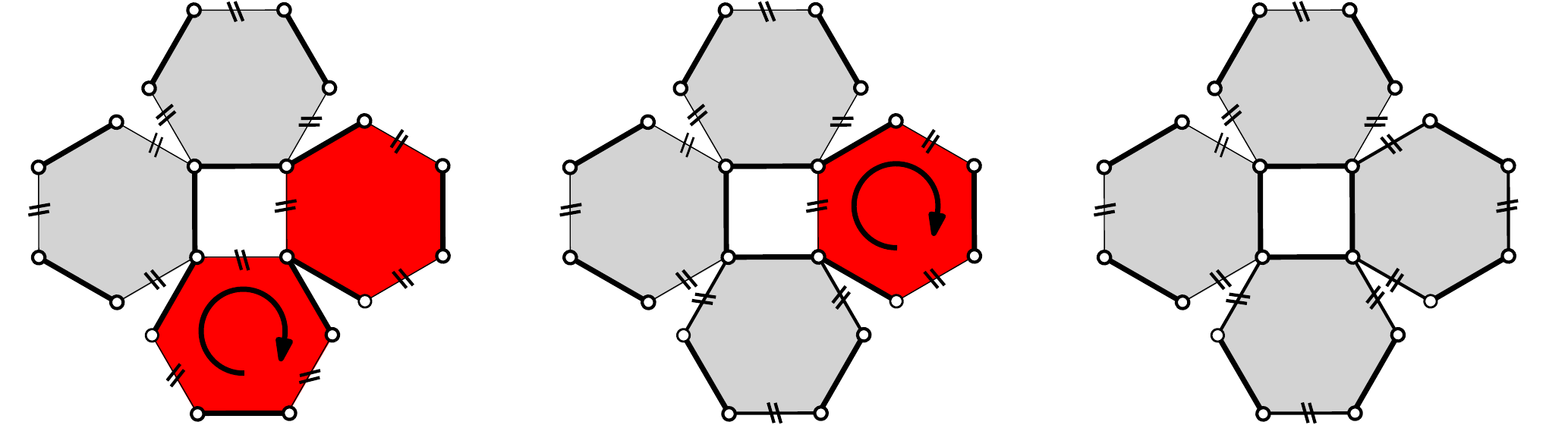}
    \end{center}
    \caption{Switching of the hexagons.}
    \label{Fig:8}
  \end{figure}

%
%
%
We intend to obtain simplifications using successive switchings for $j=1\etc n$ and using the formula
	$$ \sum_{J\subset [n]} (-1)^{n-|J|} \prod_{j\in J} \Big( \frac{d_N}N\big) \big( 1+O_j(d_N^{-\frac12})\big) = \Big( \frac{d_N}N\Big)^n \times O(d_N^{-\frac n2}),$$
valid when the $O_j(\, \cdot \,)$ depends on $j\in [n]$ but not in $J$.\\
\\ \noindent {\it Step 2: Construction of hexagons.}  We consider now the graphs of $\mcal G_N(J) \cup (w_j' \, | \,  w_j'\not \sim_J v_j)_{j\in [n]} $, to which we add labels by $2n$ more vertices $v_j'$ and $v_j''$, $j=1\etc n$:
	$$ \mcal A_N(J) = \mcal G_N(J) \cup (w_j' \, | \,  w_j'\not \sim_J v_j)_{j\in [n]} \cup (v_j' \, | \,  v_j'\sim_J w'_j)_{j\in [n]} \cup (v_j'' \, | \,  v_j'' \not \sim_J w_j)_{j\in [n]}.$$
In other words we consider the set of labeled graphs $(G,w_j', v_j', v_j'', j\in [n]),$ where
\begin{enumerate}
	\item [A1:] $G \in \mcal G_N$ is a $d_N$-regular graph on $[N]$ and $w_j', v_j', v_j'' \in [N]$ are vertices.
	\item [A2:] Four edges of each hexagon at each $j$ are constructed: 
		$$ v_j \sim_J w_j, \  w_j'\not \sim_J v_j, \ v_j'\sim_J w'_j, \ v_j'' \not \sim_J w_j.$$
\end{enumerate}
Note that the labels $v_j,w_j$ are determined by $T$. On the contrary, the $w_j', v_j', v_j''$ vary in the above enumeration. Choosing $v_j'$ and $v_j''$, we have a total of $(N-1-d_N)^n d_N^n$ possibilities (independently of the value of $J$).

Hence, one has
	\eqa
	 \mathfrak C_N(T)  := \sum_{J \subset [n]}(-1)^{n-|J|} \frac{  \sharp \mcal A_N(J)}{ \sharp G_N \times d_N^n (N-1-d_N)^n  (N-1)^n }.
		\label{Eq:Proof:01}
	\qea
\par In this definition, some of the vertices $w_j', v_j'$ or $v_j''$ can be equal and/or coincide with vertices $v_i,w_i$ for $i \in [n]$. This may be a problem for using the switching method later. Given a subset $\tilde J$ of $[n]$, we denote by $\mcal A_N(J, \tilde J)$ the set of $(G,w_j', v_j', v_j'', j\in [n]) \in \mcal A_N(J)$ such that:
	\begin{enumerate}
		\item [A3:] $j \in \tilde J$ if and only if the vertices $v_j,w_j, w_j', v_j', v_j''$ are pairwise distinct and no vertex among $w_j', v_j', v_j'' $ belongs to $ \{v_i, w_i, w_i', v_i', v_i'' \}_{i \neq j}$.
	\end{enumerate}
In other words, the set $\tilde J$ denotes the set of indices for which the vertices $ w_j', v_j', v_j''$ arise only once as labels (and do not arise as vertices of $T$). We now fix a set $\tilde J \subset [n]$ and finish the construction of the hexagons for indices in $\tilde J$ only. The contribution of the indices in $[n]\setminus \tilde J$ will be considered at the very end of our reasoning. 
 Without loss of generality, we assume that $\tilde J=\{1, \ldots, \tilde n\}$ for some $0\leq \tilde n \leq n$. Note that the choice of $J\subset [n]$ can then be decomposed as a choice for a subset of $[\tilde n]$ and of a subset of $[n]\setminus [\tilde n]$. We fix the second choice for the moment: given $J \subset [\tilde n]$, we set
	$$\mcal A_N(J, \tilde n) := \cup_{J_2 \subset{ [n]\setminus [\tilde n]}}\mcal A_N(J \cap J_2, \tilde J)$$
and prove the estimate for
	\eqa
	 \mathfrak C_N(T, \tilde n)  := \sum_{J \subset [\tilde n]}(-1)^{\tilde n-|J|} \frac{  \sharp \mcal A_N(J, \tilde n)}{d_N^n (N-1-d_N)^n (N-1)^n \sharp G_N }.
	\qea
 
\par For each graph of $\mcal A_N(J, \tilde n)$ and each $j \in [\tilde n]$, we want to add the last vertex $w_j''$ to close the associated hexagon (we do not add these vertices for indices $j\notin [\tilde n]$). This is not alway possible, given a labelled graph in $\mcal A_N(J,\tilde n)$ to add such a vertex since it may happen that $v_j'$ and $v_j''$ have $d_N$ neighbors in common. \\
Let $\hat J\subset [\tilde n]$. This partition represent the indices $j=1\etc \tilde n$ for which it will be possible to chose a vertex which is the neighbor of one but not both from $v_j'$ and $v_j''$. We only create hexagons for $j\in \hat J$.

For each ${j \in [\tilde n]}$, we consider integers $0 \leq L_j, M_j \leq 5n$ and $\ell_j \in \{ 0 \etc d_N - \max(L_j,M_j)\}$. These integers will count in the following the number of non allowed choices at each step, when choosing vertices $w_j''$, $j\in \hat J$. The order of the steps will mater, since the number of allowed choices evolves along the construction. In short we denote $\mbf L=(L_j)_{j \in [\tilde n]}$, $\mbf M = (M_j)_{j \in [\tilde n]}$ and $\boldsymbol \ell = (\ell_j)_{j \in [\tilde n]}$. Let $\mcal B_N(J,\tilde n,\hat J, \mbf L,\mbf M,\boldsymbol\ell)$ be the set of all $(G,w_j', v_j', v_j'', w_i'' j \in [n], i \in \hat J)$ such that 
\begin{enumerate}
	\item [B1:] $(G,w_j', v_j', v_j'', j \in [n])\in \mcal A_N(J, \tilde n)$, $w_j''\in [N]$ are vertices for any $j \in \hat J$.
	\item [B2:] The two remaining edges of each hexagon are chosen: $\forall j \in \hat J$,  
	$w_j''  \not \sim_Jv'_j,  w_j'' \sim_Jv''_j, $
	\item [B3:] The construction is injective: $\forall j \in \hat J$, the $w_j'' $ are pairwise distinct and do not belong to $ \{v_i, w_i, w_i', v_i', v_i'' \}_{i \in [n]}$.
	\item [B4:] Some parameters are fixed and finite: $\forall j \leq \tilde n$, $v_j'$ has exactly $L_j$ neighbors in $\{v_i,w_i, v_i', w_i',w_i''\}_{i \in [n]} \sqcup \{w_i''\}_{i<j, i\in \hat J} \setminus \{w_j'\}$. 
	\item [B5:] $\forall j  \leq \tilde n$, $v_j''$ has exactly $M_j$ neighbors in $\{v_i,w_i, v_i', w_i',w_i''\}_{i \in [n]} \sqcup \{w''_i\}_{i<j, i\in \hat J} \setminus \{w_j\}$. 
	\item [B6:] Some parameters are fixed and possibly large: $\forall j \leq \tilde n$, $v_j'$ and $v_j''$ have exactly $\ell_j$ neighbors in common out of $\{v_i,w_i, v_i', w_i',w_i''\}_{i\in[n]} \sqcup \{w_i''\}_{i<j, i\in \hat J}$.
	\item [B7:] With boundary terms: $\forall j \leq \tilde n$ such that $j\not \in \hat J$, $\ell_j = d_N-\max(L_j,M_j)$. 
\end{enumerate}
See Figure \ref{Fig:Schema1} for a scheme of this labelled graph.

  \begin{figure}[h!]
    \begin{center}
     \includegraphics[width=90mm]{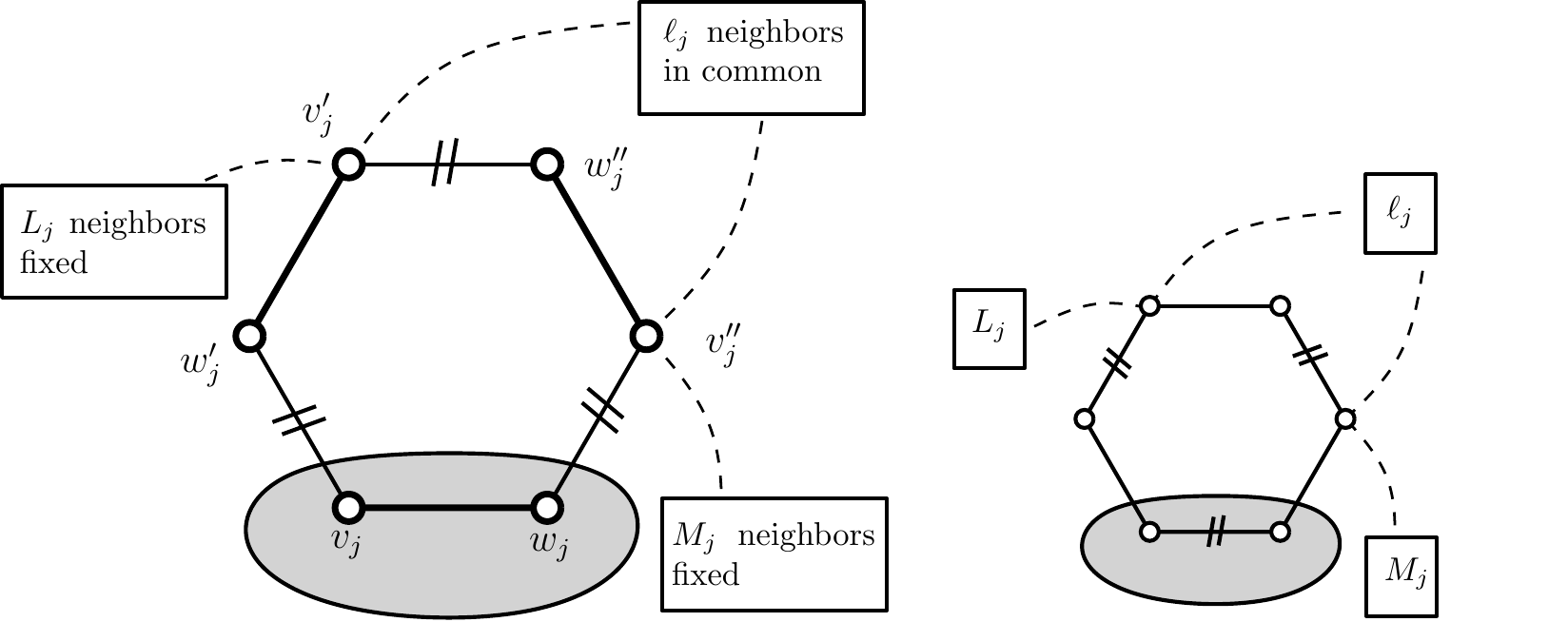}
    \end{center}
    \caption{Construction of the hexagons: detail at $j\in \hat J$ of an element of $\mcal B_N(J,\tilde n,\hat J, \mbf L,\mbf M,\boldsymbol\ell)$. This left picture represent the case where $j \in J$, the right one when $j\not \in J$.}
    \label{Fig:Schema1}
  \end{figure}

  \begin{figure}[h!]
    \begin{center}
     \includegraphics[width=110mm]{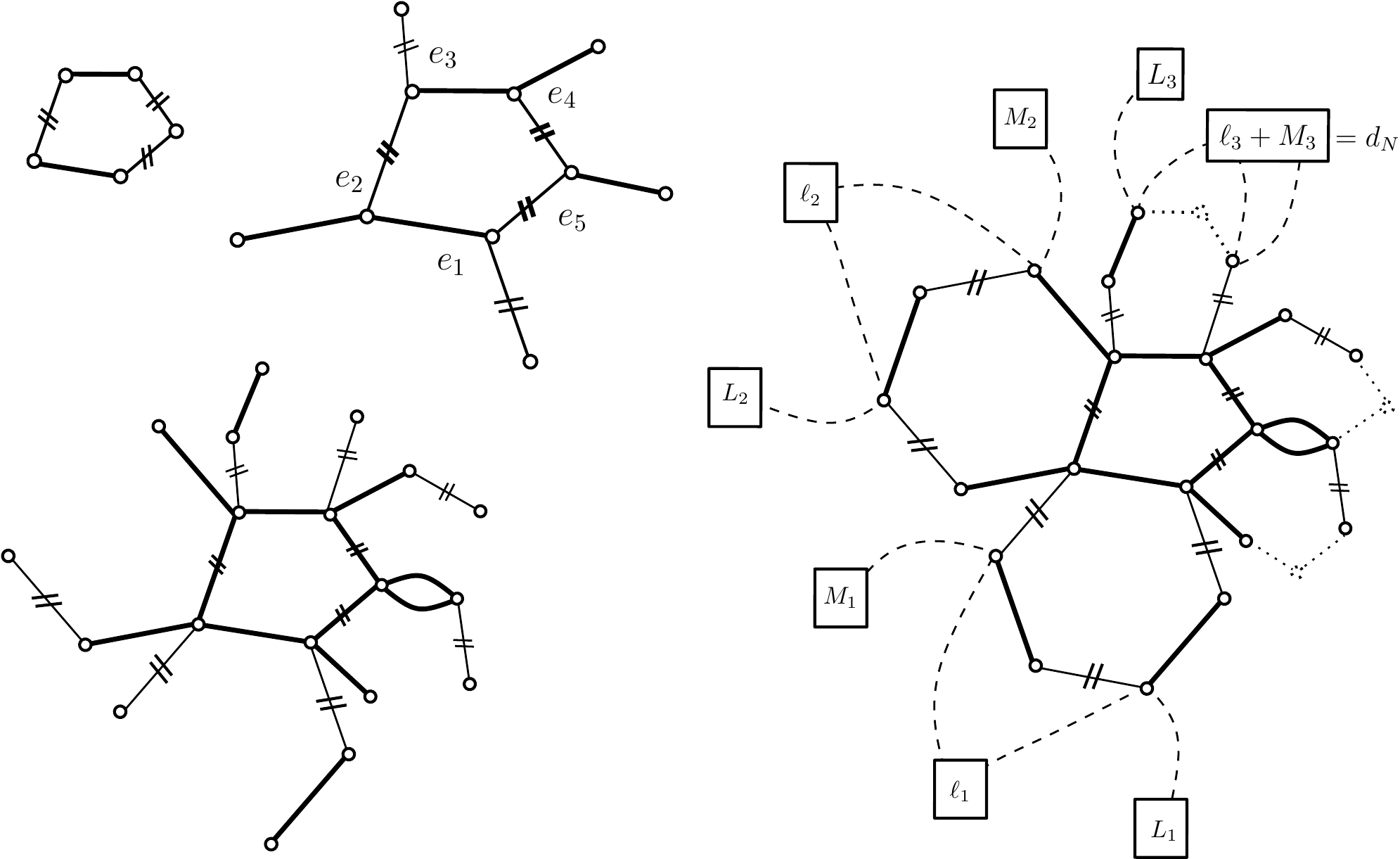}
    \end{center}
    \caption{Construction of the hexagons: an element of $\mcal B_N(J,\tilde n,\hat J, \mbf L,\mbf M,\boldsymbol\ell)$, $n=5$, $J=\{1,3\}$, $\tilde n=3$, $\hat J=\{1,2\}$.}
    \label{Fig:Schema7}
  \end{figure}

~
We now give a formula for $\sharp  \mcal A_N(J, \tilde n)$ in terms of the $\sharp \mcal B_N(J,\tilde n,\hat J, \mbf L,\mbf M,\boldsymbol\ell)$'s. Start with a graph $G$ in $\mcal A_N(J, \tilde n)$, for which the vertices $w_j', v_j'$ and $v_j"$ are chosen. Assume $1 \in \hat J$, so that we need to add the vertex $w_1''$. We partition $ \mcal A_N(J, \tilde n)$ into subsets according to the value of the constants $L_1,M_1$ and $\ell_1$ given by B4, B5 and B6. If $1\in J$, we chose $w_1''$ as a neighbor of $v_1''$ (a priori $d_N$ choices), with the restrictions that $w_1''$ is not in the vertices previously considered (remove $M_1$ choices) and is not a neighbor of $v'_1$ (remove $\ell_1$ choices). If $1\notin J$, we chose similarly $v_1''$, now among the neighbors of $v_1'$, with the same restrictions, then replacing the constant $M_1$ by $L_1$. There are either $(d_N-M_1-\ell_1)$ or $(d_N-L_1-\ell_1)$ possible choices depending on the two above cases ($1\in J$ or not). Now, we use the same procedure in order to choose the vertex $w_2''$ in $2\in \hat J$, decomposing again our set of graphs thanks to $L_2,M_2$ and $\ell_2$. If at some point $j\not \in \hat J$ we do not choose $w_j''$ and look at $j+1$ continue this process up to $\tilde n$ times. This gives
	\eqa\label{Eq:Formule B}
		\sharp \mcal A_N(J, \tilde n) & = & \sum_{\hat J \subset [\tilde n]} \sum_{\mbf L, \mbf M} \sum_{\boldsymbol \ell} \frac{ \sharp \mcal B_N(J,\tilde n,\hat J,\mbf L,\mbf M,\boldsymbol\ell) }{\kappa_N(J) } .
	\qea
where
		$$\kappa_N (J)= \prod_{ j \in   \hat J \cap J} \big(d_N - L_j - \ell_j \big)   \times \prod_{ j \in \hat J \setminus J} \big(d_N- M_j - \ell_j \big).$$

We assume without loss of generality $\hat J=\{1\etc \hat n\}$ for $1 \leq \hat n\leq \tilde n$. As before, we fix the choice of the indices of $J$ that are not in $[\hat n]$ for the moment. Given $J\subset[\hat n]$, we set now 
	$$\mcal B_N(J,\tilde n,\hat n, \mbf L,\mbf M,\boldsymbol\ell) := \cup_{J_2 \subset{ [\tilde n]\setminus \hat J}}\mcal B_N(J \cup J_2,\tilde n,\hat J, \mbf L,\mbf M,\boldsymbol\ell)$$ 
(we do not care about of the configuration at index $\hat n < j< \tilde n$) and  define
\eqa\label{Eq:Formula BB}
	 \mathfrak C_N(T, \tilde n, \hat n, \mbf L, \mbf M)  := \sum_{\ell_j, \, j \in \hat J} \  \sum_{J \subset [\hat n]}(-1)^{\hat n-|J|} \frac{  \sharp\mcal B_N(J,\tilde n,\hat n, \mbf L,\mbf M,\boldsymbol\ell)}{ \kappa_N(J) \times \sharp \mcal G_N\times d_N^n (N-1-d_N)^n(N-1)^n  }  
	\qea
where the first sum is over all $\ell_j=j\etc d_N-\max(L_j,M_j)-1$ for $j\leq \hat n$ (for $\hat n +1 \leq j < \tilde n$, there is the restriction $\ell_j=d_N-\max(L_j,M_j)$). Since there is a finite number of choices for $\hat n, \mbf L$ and $\mbf M$, it is sufficient to prove estimate \eqref{Eq:Purpose1} for that quantity.

~\\
\\\noindent{\it Step 3: first switchings.} The switching method rotates the hexagon formed by the vertices $\{ v_j, w_j, v_j', w_j', v_j", w_j"\}$ by one unit (edge) clockwise. Applied for each $j\in \hat J$, it turns a graph not containing $e_j=(v_j, w_j)$ into a graph that contains $e_j$ (changing the rightmost into the leftmost configurations in Figure \ref{Fig:Schema1}). It is important to note that the constants $\mbf L, \mbf M, \boldsymbol \ell$ remain the same after switching. This follows from the definition of the $L_j,M_j, \ell_j$'s.
It follows that the cardinal of $ \mcal B_N(J,\tilde n, \hat n,\mbf L,\mbf M,\boldsymbol\ell)$ does actually not depend on $J$ and is equal to the one of
	$$ \mcal B_{0,N}(\tilde n, \hat n, \mbf L,\mbf M,\boldsymbol\ell):=\mcal B_N([\hat n],\tilde n, \hat n, \mbf L,\mbf M,\boldsymbol\ell)$$
 for which each $j=1\etc \tilde n$ is in the leftmost configuration on Figure \ref{Fig:Schema1}.  Moreover, 
	$$ \sum_{J \subset [\hat n]}(-1)^{\hat n-|J|} \frac{ 1}{ \kappa_N(J) } = \prod_{j=1}^{\hat n}\bigg(   \frac 1 {d_N-L_j-\ell_j} - \frac 1 {d_N-M_j-\ell_j} \bigg) 	= \prod_{j \in [\hat n]}O\Big((d_N - \ell_j)^{-2} \Big),$$
which yields
	\eqa\label{Eq:FinStep1}
		\big| \mathfrak C_N( T, \tilde n, \hat n, \mbf L, \mbf M) \big| \leq \sum_{\ell_1\etc \ell_{\hat n}} \prod_{j=1}^{\hat n}O\Big((d_N - \ell_j)^{-2} \Big)\frac{  \sharp\mcal B_{0,N}(J,\tilde n, \hat n,\mbf L,\mbf M,\boldsymbol\ell)}{ d_N^n (N-1-d_N)^n(N-1)^n \sharp G_N }.
	\qea
	
The conclusion of our construction so far is that we have gained a factor $(d_N - \ell_j)^{-1}$ for each $j\in [\tilde n]$, compared with the crude estimate where we forget the condition than $w_j'' \not \sim v_j'$ and count the number of choices for $w_j''$ (see Step 4 below). We first show that we get the desired factor $d_N^{-\frac 1 2}$ provided $\ell_j$ is not to close to $d_N$. We fix a constant $0<\eta'<\eta$, where $\eta$ has been chosen in the assumption of Theorem \ref{Prop:eps_N}. 
We thus first consider the case where  $\ell_j\leq k_N:=d_N-\eta' \sqrt{d_N}$. 
 In a second time, we somehow prove that when two given vertices have a large number of neighbors in common we get an exponentially small factor (the same argument applies for indices $\hat n< j < \tilde n$).
\par We fix $J \subset [\hat n]$ and the values of $M_j,L_j, \ell_j$ for $j\leq\tilde n$. We first show how to suppress all the constructions for $j\leq \hat n$ in Steps 5 and 6 below, proving the

\begin{Lem}\label{Lem:Graph1}
		$$  \mathfrak C_N(\tilde n, \hat n,\mbf L,\mbf M)
		\leq O\big(d_N^{-\frac{\hat n }2}\big) \times \sharp \mcal C_N(\tilde n, \hat n,\mbf L,\mbf M),$$
where $\mcal C_N(\tilde n, \hat n,\mbf L,\mbf M)$ is the set of labelled graphs $(G,w_j', v_j', v_j'', \hat n < j \leq n)$ as in $
\mcal B_N(\tilde n, \hat n,\mbf L,\mbf M)$ but where do not perform the constructions of labels for indices $j\leq \hat n$. 
\end{Lem}

In steps 6 and 7 respectively, we investigate the cases where $\hat n < j< \tilde n $ and $\tilde n < j \leq n$ respectively, proving
\begin{Lem}\label{Lem:Graph2}
	$$ \frac{\mcal C_N(\tilde n, \hat n,\mbf L,\mbf M,\boldsymbol\ell)}{ \big(d_N (N-1-d_N)(N-1)\big)^{n-\hat n}}=  \sum_{J \subset [n]\setminus [\hat n]}  O\big(d_N^{-\frac{(n-\hat n )}2}\big) \times \frac{\sharp  \mcal G_N(J) \times \big( \frac {d_N}N\big)^{|J|}}{\sharp \mcal G_N}$$
\end{Lem}
Assume Lemmas \ref{Lem:Graph1} and \ref{Lem:Graph2} hold true. Then, using the trivial bound $\sharp  \mcal G_N(J)/\sharp  \mcal G_N \leq \mbb P(    T_J \subset G_N)$ where $T_J$ is the graph formed by the edges $e_j, j\in J$ of $T$, we have
	\eqa\label{Eq:Purpose2}
		|\mathfrak C_N(T) | \leq  \sum_{\tilde T \subset T} \big( \frac {d_N}N\big)^{n} \eps_N( \tilde T) \times \mbb P(   \tilde T \subset G_N), \ \  \eps_N(   \tilde T) =  O( d_N^{-\frac n2}).
	\qea
with the sum bears over the subgraphs $\tilde T$ of $T$ with exactly $n$ edges.  With the same computation as in the proof of [ref], we deduce that $P( \tilde T \subset G_N) = \alpha_N^n \big( 1 + o(d_N^{-\frac 1 2}) \big)$ for any $ \tilde T$. As a feed-back, \eqref{Eq:Purpose2} yields $\mathfrak C_N(T) = \big( \frac {d_N}N\big)^{n} \times O( d_N^{-\frac n2})$, which is \eqref{Eq:Purpose1}.
\\
\\\noindent {\it Step 4: analysis for $\ell$ smaller that $k_N$.} For simplicity, we denote $ \mcal B_N(J,\tilde n, \hat n,\mbf L,\mbf M,\boldsymbol\ell)$ by $\mcal B_N( \, \cdot \, )$  and $\ell_{\hat n}, M_{\hat n},v_{\hat n} \dots $ by $\ell,L ,v,\dots $  in the computations below. Recall that $v''$ is a neighbor of $w'$ or (exclusively) of $w''$. Using the notation $ \mcal B_N( \, \cdot \, )\setminus (w'')$ to mean that we forget the construction of $w''$ in $\mcal B(\, \cdot \,)$, we then have 
	\eq
		\sum_{\ell  = 0}^{k_N} \frac{ \sharp \mcal B_N( \, \cdot \, )}{(d_N-\ell )^2} \leq \sum_{\ell  = 0}^{k_N} \frac{ \sharp \mcal B_N( \, \cdot \, )\setminus(w'')}{d_N-\ell }.
	\qe
Similarly, with the notation $ \mcal B_N( \, \cdot \, )\setminus(v'', \ell, M, L)$ meaning that we also forget the constraints B4, B5 and B6 for $j=\tilde n$, the following inequality 
	\eq
	\sum_{\ell  = 0}^{k_N} \frac{ \sharp \mcal B_N( \, \cdot \, )}{(d_N-\ell )^2} \leq C \frac{\sharp \mcal B_N( \, \cdot \, )\setminus(w'', \ell , M , L )}{d_N-k_N},
	\qe
is valid for a constant $C$ that bounds the number of choices for the $M$ and $L$. Forgetting now the construction of the labels $v', v'' $ and $w'$ yields
	\eqa\label{Eq:CutGraphs0}
		\sum_{\ell = 0}^{k_N} \frac{ \sharp \mcal B_N( \, \cdot \, )}{(d_N-\ell)^2 d_N (N-d_N)(N-1)} \leq C \times \frac{\sharp \mcal B_N( \, \cdot \, )\setminus(\tilde n)}{d_N-k_N},
	\qea
where the notation $ \sharp \mcal B_N( \, \cdot \, )\setminus(\tilde n)$ precisely means that we do not perform any choice of labels for the index $j=\tilde n$.  We now choose $k_N=d_N-\eta'\sqrt{d_N}$. As a consequence, we get the final estimate of this step
	\eqa
		\sum_{\ell = 0}^{k_N} \frac{ \sharp \mcal B_N( \, \cdot \, )}{(d_N-\ell)^2 d_N (N-d_N)(N-1)} \leq \frac{C}{\eta'} \times \frac{\sharp \mcal B_N( \, \cdot \, )\setminus(\tilde n)}{\sqrt{d_N}}.
	\qea

\noindent{\it Step 5: analysis for $\ell$ larger than $k_N= d_N-\eta'\sqrt{d_N}$}. 
We here show 
\begin{Lem}
	There exists $0<\rho<1$ such that
\begin{equation}\label{lgrand}
\sum_{\ell \geq k_N}\frac{ \sharp\mcal B_N(J,\tilde n, \hat n,\mbf L,\mbf M,\boldsymbol\ell)}{(d_N-\ell)^2d_N(N-d_N) (N-1)}\leq \rho^{(\eta-\eta')\sqrt{d_N}}  \sharp \mcal B(\, \cdot \, )\setminus (\tilde n),
\end{equation}
where $\eta$ is the constant of Theorem \ref{Prop:eps_N}.
\end{Lem}

To this aim, we use again the switching method. We now use the notation $\mcal B(\, \cdot \,, \ell)$ for $ \mcal B_N(J,\tilde n, \hat n,\mbf L,\mbf M,\boldsymbol\ell)$ to make apparent the dependance in $\ell$. In order to compare $\mcal B(\, \cdot \,, \ell)$ with $\mcal B(\, \cdot \,, \ell-1)$ for $\ell$ bigger than a quantity $m_N:=d_N-\eta\sqrt{d_N}<k_N$,  we introduce a square whose vertices are $w'$ and new labeled vertices $s$, $s'$ and $t$.
Consider a graph $G$ in $\mcal B(\, \cdot \,, \ell)$. There are at least $\ell-6n$ possible choices of a vertex $s$,  which is a common neighbor of $w'$ and $w''$ and which is out of the set $\{v_i,w_i, v_i', w_i',w_i''\}_{i\in[n]} \sqcup \{v_i''\}_{i<j}$. Hence we deduce that
	\eq
		 \sharp \mcal B_N( \, \cdot \, , \ell) \leq \frac 1 {\ell - 6n}  \sharp \mcal B_N( \, \cdot \, , \ell)\cup( s \, | \, s\sim w' \trm{ and } s \sim w'')
	\qe
where the symbol $\mcal B_N( \, \cdot \,, \ell)\cup(s \, | \, s\sim w' \trm{ and } s \sim w'')$ represents the fact that $s$ has been chosen according to the above constraints. This notation will be used hereafter while we introduce other vertices. \\
We then choose a second vertex $s'$, among those which are not connected to $w'$ neither $w''.$ Let us count the number of possible choices.
There exist $N-d_N$ vertices which are not neighbor of $w'$. Moreover, $w''$ has $d_N-\ell$ neighbors which are not neighbors of $w'$. So the number of vertices which are neither a neighbor of $w'$ or $w''$ is at least $N-2d_N+\ell$. But $\ell$ is larger than $m_N$. Then this number is at least $N-d_N - \eta\sqrt{d_N}$. We then choose $s'$ which is neither a neighbor of $w'$ or $w''$, different from $s$ and  out of the set $\{v_i,w_i, v_i', w_i',w_i''\}_{i\in[n]} \sqcup \{v_i''\}_{i<j}$. At last, we chose a neighbor $t$ of $s'$, again out of the other vertices previously constructed. We then get
	\eqa\label{Eq:BigB}
		 \sharp \mcal B_N( \, \cdot \,, \ell ) = \frac { \sharp \mcal B_N( \, \cdot \, , \ell)\cup( s \, | \, s\sim w', s \sim w'') \cup ( s'   \, | \,  s' \not\sim w', s' \not\sim w'') \cup ( t  \, | \, , t\sim s')}{(\ell - c) (N-d_N - \eta\sqrt{d_N}-c)(d_n-c)} 
	\qea
where $c$ depends only on $n$. See Figure \ref{Schema2.pdf} for a scheme of the set in the right hand side above.

  \begin{figure}[h!]
    \begin{center}
     \includegraphics[width=110mm]{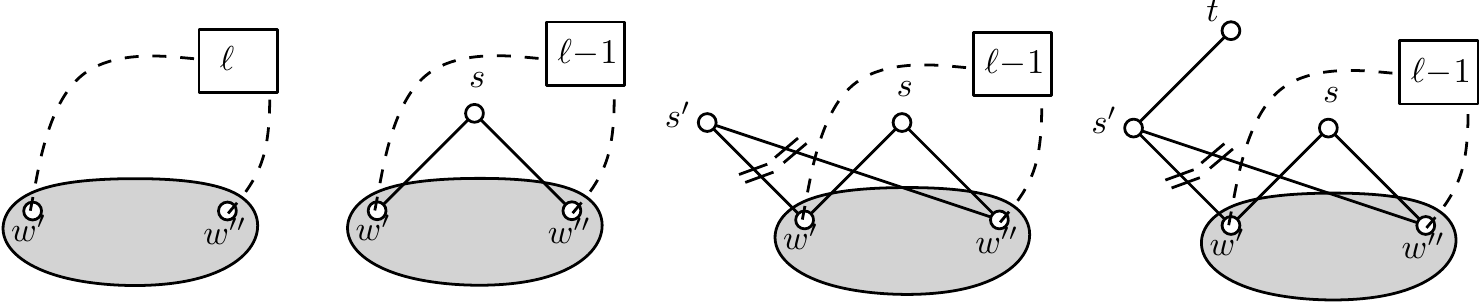}
    \end{center}
    \caption{Notations for Formula \eqref{Eq:CutGraphs}. Left: a scheme of the set in the right hand side of \eqref{Eq:BigB}. Right: the two different possibilities for the bond $t \sim s$.}
    \label{Schema2.pdf}
  \end{figure}

From the dichotomy that either $t$ and $s$ are or not connected by an edge, 
we obtain that 
	\eqa\label{Eq:CutGraphs}
		\lefteqn{  { \sharp \mcal B_N( \, \cdot \, , \ell)\cup( s \, | \, s\sim w', s \sim w'') \cup ( s'   \, | \,  s' \not\sim w', s' \not\sim w'') \cup ( t  \, | \, , t\sim s')}}\\
		& = &  { \sharp \mcal B_N( \, \cdot \, , \ell)\cup( s \, | \, s\sim w', s \sim w'') \cup ( s'   \, | \,  s' \not\sim w', s' \not\sim w'') \cup ( t  \, | \, , t\sim s',  t\not \sim s)}\nonumber\\
		& & + \  { \sharp \mcal B_N( \, \cdot \, , \ell)\cup( s \, | \, s\sim w', s \sim w'') \cup ( s'   \, | \,  s' \not\sim w', s' \not\sim w'') \cup ( t  \, | \, , t\sim s',  t  \sim s)}\nonumber
	\qea

By the switching method, for the first quantity of the right hand side in \ref{Eq:CutGraphs}, we can exchange the symbols $\sim $ and $\not \sim$ for the square formed by the sequence of vertices $(w',s',t,s)$ as long as we decrease $\ell$ by one:
	\eq
		\lefteqn{  \sharp \mcal B_N( \, \cdot \, , \ell)\cup( s \, | \, s\sim w', s \sim w'') \cup ( s'   \, | \,  s' \not\sim w', s' \not\sim w'') \cup ( t  \, | \, , t\sim s',  t\not \sim s)}\\
		& = & \sharp \mcal B_N( \, \cdot \, , \ell-1)\cup( s \, | \, s\not \sim w', s \sim w'') \cup ( s'   \, | \,  s'  \sim w', s' \not\sim w'') \cup ( t  \, | \, , t\not \sim s',  t \sim s).
	\qe
It is important to note that we do change the value of $\ell$ exactly by one while switching, since we assumed that $s'$ is not a neighbor of $w''$.
 
 The quantity in the right hand side above is obviously bounded by the same quantity where we ignore the condition that $s' \not \sim t$. Then the number of choices for $t$ is at most $d_N$, the number of choices of $s'$ is at most $d_N - \ell$ ($s'$ is now neighbor of $w'$ but not of $w''$). Counting the number of possible choices for $s$, which is a neighbor of $w''$ but not of $w'$ yields
	\eqa
			\lefteqn{   \sharp \mcal B_N( \, \cdot \, , \ell-1)\cup( s \, | \, s\not \sim w', s \sim w'') \cup ( s'   \, | \,  s'  \sim w', s' \not\sim w'') \cup ( t  \, | \, , t\not \sim s',  t \sim s)}  \nonumber\\
		& \leq &   \sharp \mcal B_N( \, \cdot \, , \ell-1)\cup( s \, | \, s\not \sim w', s \sim w'') \cup ( s'   \, | \,  s'  \sim w', s' \not\sim w'') \cup ( t  \, | \, ,  t \sim s) \nonumber\\
		& \leq  & d_N(d_N-\ell-1) \ \sharp \mcal B_N( \, \cdot \, , \ell-1)\cup( s \trm{ s.t. } s\not \sim w', s  \sim w'') \nonumber\\
		& \leq & d_N (d_N-\ell-1)^2 \ \sharp \mcal B_N( \, \cdot \,, \ell-1),\label{Eq:CutGraphs2}
	\qea
This computation is summed up in Figure \ref{Fig:3}. 

  \begin{figure}[h!]
    \begin{center}
     \includegraphics[width=130mm]{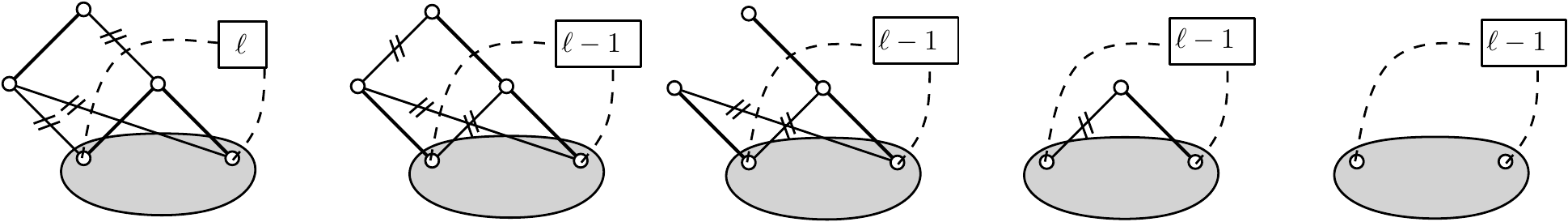}
    \end{center}
    \caption{A scheme of the proof of Estimate \eqref{Eq:CutGraphs2}.}
    \label{Fig:3}
  \end{figure}
\par Let now estimate the second quantity in the right hand side of \eqref{Eq:CutGraphs}. Using the trivial bound where we forget that $w' \not \sim s'$ and counting the choices for $s'$, $t'$ and then $s$ yields
	\eqa
		  \lefteqn{\sharp \mcal B_N( \, \cdot \, , \ell)\cup( s \, | \, s\sim w', s \sim w'') \cup ( s'   \, | \,  s' \not\sim w', s' \not\sim w'') \cup ( t  \, | \, , t\sim s',  t  \sim s)}\nonumber\\
		  & \leq & 
		   \sharp \mcal B_N( \, \cdot \, , \ell)\cup( s \, | \, s\sim w', s \sim w'') \cup ( t  \, | \, ,   t  \sim s)\cup ( s'   \, | \,  s'\sim t) \nonumber\\
		   & \leq &
		    d_N \times \sharp \mcal  B_N( \, \cdot \, , \ell)\cup( s \, | \, s\sim w', s \sim w'')  \cup ( t  \, | \, , t\sim s',  t  \sim s)\nonumber\\
		   & \leq & d_N^2 \ell \ \sharp \mcal B(\, \dot \,, \ell).\label{Eq:CutGraphs3}
	\qea

This computation is summed up in Figure \ref{Fig:4}. 

  \begin{figure}[h!]
    \begin{center}
     \includegraphics[width=130mm]{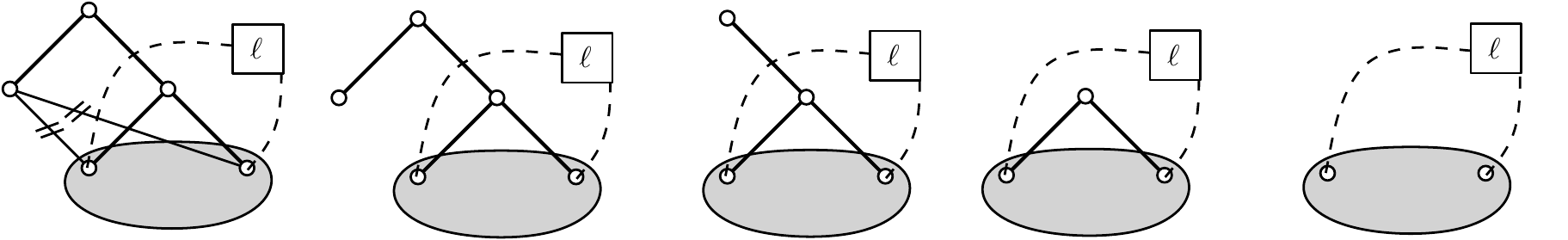}
    \end{center}
    \caption{A scheme of the proof of Estimate \eqref{Eq:CutGraphs3}.}
    \label{Fig:4}
  \end{figure}

To sum up, by \eqref{Eq:CutGraphs}, \eqref{Eq:CutGraphs2} and \eqref{Eq:CutGraphs3}, there exists a constant $c$ such that 
	$$ (N-d_N-\eta\sqrt{d_N}-c)(d_N-c) \ \sharp \mcal B_N(\, \cdot \, , \ell) \leq \frac{d_N(d_N-\ell)^2 \ \sharp \mcal B_N(\, \cdot \, , \ell-1) + d_N^2 \ell \ \sharp \mcal B_N(\, \cdot \, , \ell)}{\ell-c},$$
	or equivalently
	$$  \sharp \mcal B_N(\, \cdot \, , \ell)  \leq O \Big( \frac 1 {N-2d_N-\eta\sqrt{d_N}} \Big) \frac{(d_N-\ell)^2}\ell  \sharp \mcal B_N(\, \cdot \, , \ell-1).$$ 
But $\ell\geq d_N-\eta\sqrt d_n$ yields $ \frac{(d_N-\ell)^2}\ell  \leq \frac{\eta^2d_N}{d_N-B\sqrt d_N}\leq \eta^2$. Hence, thanks to the technical condition $N-2d_N-\eta\sqrt{d_N} \limN \infty$, for $N$ large enough there exists $\rho<1$ such that 
	$$ \sharp \mcal B_N(\, \cdot \, , \ell) \leq \rho  \ \sharp \mcal B_N(\, \cdot \, , \ell-1) \leq \rho^{\ell - m_N}\sharp \mcal B(\, \cdot , m_N)$$
where we recall that $m_N = d_N - \eta\sqrt d_N$.
With the same reasoning as in step 3, we bound $B(\, \cdot , m_N)$ from above by successively: forgetting the condition $f_j'' \not\subset G_N$, counting the number of choices for $v''$, forgetting the conditions given by $L,M$ and $m_N$, and counting the number of choices for $w'$, $v'$ and $w''$. This yields 
	$$\sharp \mcal B(\, \cdot , m_N) \leq (d_N-m_N) d_N(N-d_N)^2 \sharp \mcal B(\, \cdot \, )\setminus \{\tilde n\},$$
where we remind that $\mcal B(\, \cdot \, )\setminus \{\tilde n\}$ means we have forget all constructions for the index $j = \tilde n$ in the construction of $\mcal B(\, \cdot \, )$.

Hence, we obtain
	$$\sum_{\ell = k_N}^{d_N - \max(M,L)} \frac{ \sharp \mcal B_N( \, \cdot \, )}{(d_N-\ell)^2 d_N (N-d_N)(N-1)}  \leq \rho^{k_N-m_N} \sharp \mcal B_N( \, \cdot \, ) \setminus \{\tilde n\} = \rho^{(\eta-\eta')\sqrt{d_N}} \sharp \mcal B_N( \, \cdot \, ) \setminus \{\tilde n\}.$$
This is the announced estimate (\ref{lgrand}).\\
 The conclusion of Steps 4 and 5 is that Lemma \ref{Lem:Graph1} is true. 
\\
\\{\it Step 6: Indices $\hat n < j \leq \tilde n$.} By Step 4, it is clear that 
	$$\sharp \mcal C_N(\tilde n, \hat n,\mbf L,\mbf M) \leq \sqrt{d_N}^{-\frac{ (\tilde n - \hat n)}2} \sharp \mcal D_N(\tilde n, \hat n,\mbf L,\mbf M),$$
where $\mcal D_N(\tilde n, \hat n,\mbf L,\mbf M)$ is the set of labelled graphs $(G,w_j', v_j', v_j'', \tilde n < j \leq n)$ such that all the constructions in $\mcal C_N(\tilde n, \hat n,\mbf L,\mbf M)$ for $\hat n<j\leq \tilde n$ are forgotten. 
\\
\\{\it Step 7: Indices $\hat n < j \leq \tilde n$.} For each $j=\tilde n+1\etc n$, there exists (at least) a vertex $x_j$ among $v_j'$, $w_j'$ and $w_j''$ which coincides with (at least) another vertex in $\{v_i, w_i\}_{i\in n} \cup \{v_i', w_i', w_i''\}_{i>\tilde n}$. Assume that $j \in J$. If we forget the constraint A3, they would be $d_N$ choices for $w_j'$ and $N-1-d_N$ choices for $v_j'$ and $w_j'$. If instead $j \not\in J$, these choices are respectively $N-1-d_N$ and $d_N$ choices. \\
As a consequence, if $x_j$ is chosen in the fixed vertices $v_i,w_i, i=1\etc n$, it is clear that we gain at least $d_N^{-1}$ from that term.
Assume now that $x_j$ arises as a labeled vertex associated to $m_j-1$ other vertices among $v_i', w_i', w_i'', i>\tilde n$. In other word, $m_j$ is the multiplicity of $x_j$ as a labelled vertex. Without the constraint A3, the number of choices for $x_j$ and these $m_j-1$ other vertices is of the form $d_N^{p_j}\times (N-d_N)^{q_j}$ for $p_j+q_j=m_j$. Instead, it is less than $O(d_N)$ if $q_j\geq 1$ and less than $O(N)$ otherwise. We gain for each copy of the copies of $x_j$ at least the $m_j$-th root of its contribution $d_N$ or $N-d_N$. Hence we gain at least the expected term $\sqrt{d_N}$ for each index $j$.

\end{document}